\documentclass[12pt]{article}
\usepackage{latexsym, enumerate}
\usepackage{color}
\usepackage{amsmath, amsthm, amsfonts, amssymb, color}
\usepackage{mathrsfs}
\usepackage{cite}
\usepackage{url}

\newcommand{\scr}[1]{\mathscr #1}

\def\N{\mathbb{N}}

\def\R{\mathbb{R}}
\def\E{\mathbb{E}}
\def\EE{\mathrm{E}}
\def\PP{\mathbb{P}}

\def\W{\mathrm{W}}
\def\d{\mathrm{d}}
\def\L{\mathrm{L}}
\def\e{\mathrm{e}}
\def\I{\mathrm{I}}
\def\M{\scr M}

\def\J{\mathrm{J}}

\def\Pr{\scr P}

\def\<{\langle} \def\>{\rangle}

\newcommand{\mathbbm}[1]{\text{\usefont{U}{bbm}{m}{n}#1}}

\DeclareMathOperator{\Lip}{Lip}
\DeclareMathOperator{\diam}{diam}

\newtheorem{theo}{\textsc{Theorem}}
\newtheorem{lem}[theo]{\textsc{Lemma}}
\newtheorem{propo}[theo]{\textsc{Proposition}}
\newtheorem{cor}[theo]{\textsc{Corollary}}
\theoremstyle{remark}
\newtheorem{remark}{Remark}
\theoremstyle{definition}

\newcommand\beq{\begin {equation}}
\newcommand\eeq{\end {equation}}
\newcommand\beqs{\begin {equation*}}
\newcommand\eeqs{\end {equation*}}

\begin{document}

\title{Asymptotic behavior of Wasserstein distance for weighted empirical measures of diffusion processes on compact Riemannian manifolds}

\author{\em{Jie-Xiang Zhu}\footnote{Research supported partially by the National Key R\&D Program of China (No. 2022YFA1006000, 2020YFA0712900) and NNSFC (11921001).}\\
{\footnotesize Center for Applied Mathematics, Tianjin University, Tianjin 300072, China}\\
{\footnotesize  jiexiangzhu7@gmail.com }
}
\date{\today}

\maketitle

\begin{abstract}
Let $(X_t)_{t \ge 0}$ be a diffusion process defined on a compact Riemannian manifold, and for $\alpha > 0$, let
$$
\mu_t^{(\alpha)}  =  \frac{\alpha}{t^\alpha} \int_{0}^{t} \delta_{X_s} \, s^{\alpha - 1} \d s
$$
be the associated weighted empirical measure. We investigate asymptotic behavior of $\E^\nu \big[ \W_2^2(\mu_t^{(\alpha)}, \mu) \big]$ for sufficient large $t$, where $\W_2$ is the quadratic Wasserstein distance and $\mu$ is the invariant measure of the process. In the particular case $\alpha = 1$, our result sharpens the limit theorem achieved in \cite{WZ19}. The proof is based on the PDE and mass transportation approach developed by L. Ambrosio, F. Stra and D. Trevisan.
\end{abstract}\noindent



\vskip 2cm

\renewcommand{\theequation}{\thesection.\arabic{equation}}
\setcounter{equation}{0} \maketitle

\vspace{-1.0cm}

\bigskip

\section{Introduction and main results}

The long-time asymptotic behavior of empirical measures is a classical topic in the fields of probability theory and statistics, and has been widely investigated in various settings. Let $M$ be a $d$-dimensional smooth complete connected compact Riemannian manifold without boundary. Let $V \in C^2(M)$ such that the measure $\d \mu  = \e^V \d x$ is a probability measure, where $\d x$ is the Riemannian measure on $M$. Let $(X_t)_{t \ge 0}$ denote the diffusion process on $M$ generated by the second order differential operator $\L  = \Delta + \nabla V \cdot \nabla$, where $\Delta$ and $\nabla$ are the Laplace-Beltrami operator and the gradient operator on $M$, respectively. Denote the associated diffusion semigroup by $( P_t )_{t \geq 0}$, which is symmetric with respect to the invariant measure $\mu$. The empirical measure on $(X_t)_{t \ge 0}$
\beq \label{em}
\mu_t \,  := \, \frac{1}{t} \int_{0}^{t} \delta_{X_s} \d s, \quad t > 0,
\eeq
characterizes the sample path of $X_t$. Under our assumptions, $\mu_t$ converges to $\mu$ as $t \to \infty$ in the weak topology. In this work, we will consider a family of weighted empirical measures on $(X_t)_{t \ge 0}$, denoted by $\mu_t^{(\alpha)}$ $(\alpha>0)$, which is defined by
\beq \label{wem}
\mu_t^{(\alpha)} \, : = \, \frac{\alpha}{t^\alpha} \int_{0}^{t} \delta_{X_s} \, s^{\alpha - 1} \d s, \quad t > 0.
\eeq
The weighted empirical measure $\mu_t^{(\alpha)}$ is equivalent to the usual empirical measure on the time-changed process $(X_{t^{\alpha^{-1}}})_{t \ge 0}$ at time $t^\alpha$. This family of weighted empirical measures can be used to approximate other kinds of empirical measures. For instance, for $N \in \N$ fixed, there exists a sequence of polynomials $\{ Q_k \}_{k = 1}^{\infty}$ such that for every $f \in C_b([0, 1])$,
\beqs
\lim_{k \to \infty} \int_{0}^{1} f(y) Q_k(y) \, \d y \, = \, \frac1N \sum_{i = 1}^{N} f(i/N).
\eeqs
And as a consequence of ergodicity, the law of $(X_{t/N}, X_{2t/N}, \ldots, X_t)$ converges to $\mu^{\otimes N}$ as $t \to \infty$. Therefore the weighted empirical measures $\frac{1}{t} \int_{0}^{t}  \delta_{X_s} Q_k(s/t) \d s$, each of which is a finite linear combinations of $\mu_t^{(\alpha)}$, simulate the empirical measure generated by $N$ independent samples with common distribution $\mu$ as $k, t \to \infty$.

It is natural to explore the convergence rate of $\mu_t^{(\alpha)}$ to $\mu$. We formulate this problem in the framework of Wasserstein (Kantorovich) distances. Let $\rho$ be the geodesic distance on $M$, and let $\Pr(M)$ be the set of all Borel probability measure on $M$. Given $p \geq 1$, the $p$-th Wasserstein distance (see e.g. \cite{V09}) between $\nu_0, \nu_1 \in \Pr(M)$ is defined by
\beqs
\W_p(\nu_0, \nu_1) \, = \, \inf_\pi \bigg( \int_{M \times M} \rho(x, y)^p \d \pi(x, y)  \bigg)^{\frac1p},
\eeqs
where the infimum is taken over all couplings $\pi$ on $M \times M$ with respective marginals $\nu_0$ and $\nu_1$. The aim of this paper is to explore the order of decay in $t$, and furthermore, the renormalized limit of the expectation
\beqs
\E^\nu \big[ \W_2^2(\mu_t^{(\alpha)}, \mu)  \big],
\eeqs
as $t \to \infty$, where $\E^\nu$ denotes the expectation taken for the diffusion process with initial distribution $\nu \in \Pr(M)$. Note that the corresponding results for $\mu_t$ (namely the case $\alpha = 1$) could be found in \cite{WZ19}.

A major progress in recent years is the work of L. Ambrosio, F. Stra and D. Trevisan \cite{AMB}, which studies the asymptotic behavior of $\E [ \W_2^2(\mu_n, \mu) ] $,
where
$$
\mu_n \,  := \, \frac{1}{n} \sum_{i = 1}^{n} \delta_{X_i}, \quad n \in \N,
$$
is the empirical measure of independent random variables $(X_i)_{1 \le i \le n}$ with common distribution $\mu$, a discrete analogue of \eqref{em}. More precisely, they proved that if $\mu$ is the Riemannian measure on a compact $2$-dimensional compact manifold $(M, \rho)$ without boundary, then
\beqs
\lim_{n \to \infty} \frac{n}{\ln n} \E \big[ \W_2^2(\mu_n, \mu) \big] \, = \, \frac{1}{4\pi}.
\eeqs
Their proof is based on a PDE and mass transportation approach. The key idea involves the comparison between the $2$-Wasserstein distance and the $H^{-1, 2}$-Sobolev norm of the density function (see Section \ref{wd and sobo} below). This methodology has already been applied to a number of settings. Here we only cite a few advances on the continuous time cases. The long-time behavior of $\E^\nu [\W_2^2(\mu_t, \mu)]$ as well as its renormalized limit in lower-dimensional case has been studied in \cite{W21, W23+, W22, W23, WZ19} for empirical measures of symmetric diffusion processes. The extension of these results could be found in \cite{WB23, LW22, LW22+, LW22++} for symmetric subordinated processes, \cite{HMT22} for fractional Brownian motion on torus and \cite{W23++} for non-symmetric subordinated processes. See also \cite{D23} for developments in the context of McKean-Vlasov SDEs.

Denote by $\{ \lambda_i \}_{i \geq 1}$ the sequence of non-trivial eigenvalues of $(-\L)$ listed in the increasing order with multiplicity, and by $\{ \phi_i \}_{i \geq 1}$ the sequence of the corresponding eigenfunctions, which is orthonormal in $L^2(\mu)$. For the weighted empirical measure \eqref{wem}, our first result is stated as follows.

\begin{theo} \label{main}
Let $(M, \rho)$ be a connected complete compact weighted Riemannian manifold without boundary, equipped with a weighted probability measure $\d \mu = \e^V \d x$. Assume $M$ has the dimension $d \le 3$. Then
\begin{enumerate}[(i)]
\item If $\alpha \in (\frac{(d - 2)^+}{4}, \frac12)$,
\beqs
\lim_{t \to \infty} t^{2\alpha} \, \E^\nu \big[ \W_2^2 ( \mu_t^{(\alpha)}, \mu )  \big] \, = \, \nu \big( H^{(\alpha)} \big),
\eeqs
where $H^{(\alpha)} \in C_b(M)$ is defined by
\beqs
H^{(\alpha)}(x) \, = \, 2 \alpha^2 \iint_{\{0 \, \leq \, u \, \leq \, v \, < \, \infty  \}}  P_{u} \bigg( \sum_{i=1}^{\infty} \lambda_i^{-1} \e^{-\lambda_i(v - u)}\phi_i^2 \bigg)(x) \, u^{\alpha - 1} v^{\alpha - 1} \d u \d v, \quad x \in M.
\eeqs

\item If $\alpha = \frac12$,
\beqs
\lim_{t \to \infty} \frac{t}{\ln t} \, \E^\nu \big[ \W_2^2 ( \mu_t^{(\alpha)}, \mu )  \big] \, = \, \frac12 \sum_{i = 1}^{\infty} \frac{1}{\lambda_i^2}.
\eeqs

\item If $\alpha \in (\frac12, \infty)$,
\beqs
\lim_{t \to \infty} t \, \E^\nu \big[ \W_2^2 ( \mu_t^{(\alpha)}, \mu )  \big] \, = \, \frac{2 \alpha^2}{2\alpha - 1} \sum_{i = 1}^{\infty} \frac{1}{\lambda_i^2 }.
\eeqs
\end{enumerate}
Moreover, all the above limits converge uniformly in $\nu \in \Pr(M)$.
\end{theo}

In order to unify notations, we introduce a function $R_\alpha : \R_+ \to \R_+$ to describe the decay rate as follows
\begin{align*}
R_\alpha (t) \, = \,
\begin{cases}
t^{-2\alpha} , & \textrm{if }\alpha \in (0, \frac12);\\
t^{-1} \ln t, & \textrm{if }\alpha = \frac12;\\
t^{-1}, & \textrm{if }\alpha \in (\frac12, \infty),
\end{cases}
\end{align*}
and we denote by $L_\alpha(\nu)$ the value of the limit presented in Theorem \ref{main} (depending on $\alpha$ and the initial distribution $\nu \in \Pr(M)$). As expected, Theorem \ref{main} quantifies the fact that the convergence of $\mu_t^{(\alpha)}$ to $\mu$ accelerates as $\alpha$ increases.  Theorem \ref{main} also indicates that the case $\alpha = \frac12$ is critical in two aspects. First, the decay rate function $R_\alpha$ (as a function of $\alpha$) has a change at the value $\alpha = \frac12$. Moreover, if $\alpha \in (\frac{(d - 2)^+}{4}, \frac12)$, the value of the renormalized limit $L_\alpha(\nu)$ depends on the initial distribution $\nu$, while if $\alpha \in [\frac12, \infty)$, $L_\alpha(\nu)$ no longer depends on $\nu$, which is a phenomenon of interest. For more quantitative information of $\E^\nu \big[ \W_2^2 ( \mu_t^{(\alpha)}, \mu )  \big]$, our argument can indeed tighten Theorem \ref{main} in the following sense.

\begin{theo} \label{main1'}
Under the assumptions of Theorem \ref{main}. Then
\beqs
\limsup_{t \to \infty} \sup_{\nu \in \Pr(M)} \bigg( R_\alpha(t)^{-1} \E^\nu \big[ \W_2^2 \big( \mu_t^{(\alpha)}, \mu \big) \big] \, - \, L_\alpha (\nu) \bigg) \cdot K(t)^{-1} \, < \, \infty,
\eeqs
and
\beqs
\liminf_{t \to \infty} \inf_{\nu \in \Pr(M)} \bigg( R_\alpha(t)^{-1} \E^\nu \big[ \W_2^2 \big( \mu_t^{(\alpha)}, \mu \big) \big] \, - \, L_\alpha (\nu) \bigg) \cdot \widetilde{K}(t)^{-1} \, > \, -\infty,
\eeqs
where $K(t)$ and $\widetilde{K}(t)$ are positive functions depending on $d$ and $\alpha$, and satisfy $\lim_{t \to \infty} K(t) = \lim_{t \to \infty} \widetilde{K}(t) = 0$.
\end{theo}
The functions $K(t)$ and $\widetilde{K}(t)$ are explicitly given at the end of Section \ref{up} and \ref{low} respectively, which might not be optimal. In the particular case $\alpha = 1$, Theorem \ref{main1'} implies that as $t \to \infty$,
\beqs
\sup_{\nu \in \Pr(M)} \bigg| t \E^\nu \big[ \W_2^2 \big( \mu_t, \mu \big) \big] \, - \,  \sum_{i = 1}^{\infty} \frac{2}{\lambda_i^2 } \bigg| \, = \,
\begin{cases}
O(t^{-\frac13} \log(t)^{-1} ), & \textrm{when } d = 1;\\
O_\epsilon(t^{-\frac14 + \epsilon}), & \textrm{when } d = 2;\\
O_\epsilon(t^{-\frac{1}{12} + \epsilon}), & \textrm{when } d = 3.
\end{cases}
\eeqs
This bound sharpens Theorem 1.1.(1) of \cite{WZ19}. Throughout this work, $c$ denotes a positive absolute constant that may only depend on the underlying manifold $M$ and the parameter $\alpha$, whose value could change from line to line. We use the notation $f = O(g)$ if there exists $c > 0$ such that $f \leq c g$. We write $f \sim g$ if both $f = O(g)$ and $g = O(f)$ hold. We also write $f = O_\epsilon(g(\epsilon))$ if for any $\epsilon > 0$ there exists $c_\epsilon > 0$, depending on $\epsilon$, such that $f \leq c_\epsilon g(\epsilon)$.

For the remaining possible values of $d$ and $\alpha$, by the PDE and mass transportation approach, we have the following upper bounds for $\E^\nu \big[ \W_2^2 \big( \mu_t^{(\alpha)}, \mu \big) \big]$.

\begin{theo} \label{main2'}
Let $(M, \rho)$ be a connected complete compact weighted Riemannian manifold without boundary, equipped with a weighted probability measure $\d \mu = \e^V \d x$. Assume $M$ has the dimension $d \geq 3$. Then for all sufficiently large $t$,
\begin{enumerate}[(i)]
\item When $d = 3$,
\beqs
\sup_{\nu \in \Pr(M)} \E^\nu \big[ \W_2^2 \big( \mu_t^{(\alpha)}, \mu \big) \big] \, = \,
\begin{cases}
O(t^{-\frac{4\alpha}{3 - 4\alpha}}), & \textrm{if } \alpha \in (0, \frac14);\\
O(t^{-\frac12}\log(t)), & \textrm{if } \alpha = \frac14.
\end{cases}
\eeqs

\item When $d = 4$,
\beqs
\sup_{\nu \in \Pr(M)} \E^\nu \big[ \W_2^2 \big( \mu_t^{(\alpha)}, \mu \big) \big] \, = \,
\begin{cases}
O(t^{-\frac{\alpha}{1 - \alpha} } \log(t)), & \textrm{if } \alpha \in (0, \frac12);\\
O(t^{-1}\log(t)^2), & \textrm{if } \alpha = \frac12;\\
O(t^{-1} \log(t)), & \textrm{if } \alpha \in (\frac12, \infty).
\end{cases}
\eeqs

\item When $d \geq 5$,
\beqs
\sup_{\nu \in \Pr(M)} \E^\nu \big[ \W_2^2 \big( \mu_t^{(\alpha)}, \mu \big) \big] \, = \,
\begin{cases}
O(t^{-\frac{4 \alpha}{d - 2} }), & \textrm{if } \alpha \in (0, \frac12);\\
O(t^{- \frac{2}{d - 2}}\log(t)), & \textrm{if } \alpha = \frac12;\\
O(t^{- \frac{2}{d - 2}}), & \textrm{if } \alpha \in (\frac12, \infty).
\end{cases}
\eeqs

\end{enumerate}

\end{theo}

Theorem \ref{main2'} shows that the case $d = 4$ is critical, which has already been noticed in \cite{WZ19}. In analogy of Theorem \ref{main2'}, the two-sided estimates of $\E^\nu \big[ \W_2^2(\mu_t, \mu) \big]$ in higher dimensions have been achieved in \cite{WZ19}, which asserts that when $d \ge 5$,
\begin{equation} \label{d5}
\E^\nu \big[ \W_2^2(\mu_t, \mu) \big] \, \sim \, t^{-\frac{2}{d-2}},
\end{equation}
and when $d = 4$, there exist constants $c_1, c_2 > 0$, depending on $M$, such that
\begin{equation} \label{d4}
c_1 t^{-1} \, \leq \, \E^\nu \big[ \W_2^2(\mu_t, \mu) \big] \, \leq \, c_2 t^{-1} \log (t)
\end{equation}
holds for sufficiently large $t$.

However, there is a gap between the lower and upper bounds of \eqref{d4}. In the work \cite{W23++}, the sharp convergence rate is explored in a more general setting, and reads as follows
\begin{equation} \label{d4'}
\E^\nu \big[ \W_2^2(\mu_t, \mu) \big] \, \sim \, t^{-1} \log (t).
\end{equation}
In this paper, we investigate the lower bound estimate in the $4$-dimensional case and obtain the following more precise result.

\begin{theo} \label{main2}
Let $(M, \rho)$ be a connected complete compact weighted Riemannian manifold without boundary, equipped with a weighted probability measure $\d \mu = \e^V \d x$. Assume $M$ has the dimension $d  = 4$. Then for any $\alpha \in (\frac12, \infty)$,
\beqs
\liminf_{t \to \infty} \frac{t}{\ln t}  \inf_{\nu \in \Pr(M)} \E^\nu \big[ \W_2^2 \big( \mu_t^{(\alpha)}, \mu \big) \big]  \, \ge \, \frac{\alpha^2}{2\alpha - 1} \frac{1}{32\pi^2}.
\eeqs
\end{theo}
For the special case $\alpha = 1$, Theorem \ref{main2} reads
\beqs
\liminf_{t \to \infty} \frac{t}{\ln t}  \inf_{\nu \in \Pr(M)} \E^\nu \big[ \W_2^2 \big( \mu_t, \mu \big) \big]  \, \ge \, \frac{1}{32\pi^2},
\eeqs
which in particular sharpens the lower bound of \eqref{d4'}. We further conjecture that there exists a universal constant $c > 0$, which does not depend on specific $4$-dimensional compact manifolds and is a rational multiple of $\pi^{-2}$, such that
\beqs
\lim_{t \to \infty} \frac{t}{\ln t} \E^\nu \big[ \W_2^2 \big( \mu_t, \mu \big) \big]  \, = \, c.
\eeqs

This paper is organized as follows. Section \ref{S0} provides several properties of the heat semigroup and Wasserstein distance. Sections \ref{S1} and \ref{s3} develop two types of estimates: Section \ref{S1} uses $L^2$-structure, while Section \ref{s3} relies on concentration inequalities. We prove Theorems \ref{main} and \ref{main1'} in Section \ref{s4} by dealing with upper and lower bounds separately. The proofs of Theorems \ref{main2'} and \ref{main2} are presented in Section \ref{s5}.

\section{Preliminaries} \label{S0}
\setcounter{equation}{0}

In this section, we collect some basic properties of the heat semigroup and Wasserstein distance, which will be used throughout this work.

\subsection{Heat semigroup on $M$}

The materials presented in this sub-section can be found in \cite{BGL14, W14}. Let $\nabla$ denote the gradient operator on $M$. Since $M$ is a $d$-dimensional compact manifold, there exists $K \in \R$ such that for any smooth function $f : M \to \R$,
\beq \label{cd condition}
\frac12 \L |\nabla f|^2 - \< \nabla \L f, \nabla f  \> \, \geq \, K |\nabla f|^2 + \frac1d (\L f)^2.
\eeq
This inequality is known as curvature-dimension condition $CD(K, d)$. Throughout this work, we denote by $\| \cdot \|_p$ the norm of $L^p(\mu)$. And for all $1 \leq p, q \leq \infty$, we denote by $\| \cdot \|_{p \to q}$ the operator norm from $L^p(\mu)$ to $L^q(\mu)$. Under the condition $CD(K, d)$, $(P_t)_{t \geq 0}$ is ultracontractive. More precisely, we have
\beq \label{sobolev}
\| P_t \|_{1 \to \infty} \, \leq \, c  (1 \wedge t)^{-\frac{d}{2}}, \quad t > 0.
\eeq
Let $p_t$ be the heat kernel of $P_t$ w.r.t. $\mu$, i.e. for any $f \in C_b(M)$,
$$
P_t f (x) \, = \, \E^x \big[  f(X_t)  \big] \, = \, \int_M f(y) p_t(x, y)  \, \d \mu (y), \quad t > 0, \, \, x \in M.
$$
Then \eqref{sobolev} is equivalent to the upper bound on $p_t$ as follows:
\beq \label{upper pt}
\sup_{x, y \in M} p_t(x, y) \, \leq \, c  (1 \wedge t)^{-\frac{d}{2}}, \quad t > 0.
\eeq
Since $(P_t)_{t \geq 0}$ is contractive on $L^p(\mu)$ for any $1 \leq p \leq \infty$, by interpolation, there exists $c > 0$ such that for any $1 \leq p \leq q \leq \infty$,
\beq \label{general sobolev}
\| P_t \|_{p \to q} \, \leq \, c  (1 \wedge t)^{-\frac{d}{2}(p^{-1} - q^{-1})}, \quad t > 0.
\eeq

As another corollary of $CD(K, d)$, the spectrum of $\L$ is discrete. Denote by $\{ \lambda_i \}_{i \geq 1}$ the sequence of non-trivial eigenvalues of $(-\L)$ listed in the increasing order counting multiplicity, and by $\{ \phi_i \}_{i \geq 1}$ the sequence of the associated eigenfunctions, which is orthonormal in $L^2(\mu)$. Then $p_t$ has the spectral representation as follows:
\beq \label{eigen rep}
p_t(x, y) \, = \, 1  \, + \, \sum_{i = 1}^{\infty} \e^{-\lambda_i t} \phi_i(x) \phi_i(y), \quad t > 0, \, \, x, y \in M.
\eeq
Given $\theta \geq 0$, for $r >0$, we introduce the function $h^{(\theta)}_r : M \to \R_+$ as follows
\beq \label{theta function}
h^{(\theta)}_r(x) \, : = \, \sum_{i = 1}^{\infty} \lambda_i^{- \theta} \e^{- \lambda_i r} \phi_i^2(x) , \quad  x \in M.
\eeq
For later purposes, we also set
\beq \label{theta function'}
\overline{h}^{(\theta)}_r \, := \, \mu (h^{(\theta)}_r) \, = \, \sum_{i = 1}^{\infty} \lambda_i^{- \theta} \e^{- \lambda_i r}, \quad g^{(\theta)}_r \, := \, - \partial_\theta \overline{h}^{(\theta)}_r \, = \, \sum_{i = 1}^{\infty} \ln(\lambda_i) \lambda_i^{- \theta} \e^{- \lambda_i r}.
\eeq
Then \eqref{upper pt} together with \eqref{eigen rep} implies the following upper bounds.
\begin{lem} \label{trace}
Given $\theta \ge 0$, there exists $c = c(\theta) > 0$ such that for any $x \in M$ and $r > 0$,
\begin{align}
h^{(\theta)}_r(x) \, \leq \, c \e^{-\lambda_1 r} \cdot
\begin{cases} \label{bound1}
(1 \wedge r)^{-(\frac{d}{2} - \theta)}, & \textrm{if } 0 \le \theta < \frac{d}{2};\\
\log((1 \wedge r)^{-1} + 1), & \textrm{if } \theta = \frac{d}{2};\\
1, & \textrm{if } \theta >  \frac{d}{2}.
\end{cases}
\end{align}
Since $\overline{h}^{(\theta)}_r = \mu (h^{(\theta)}_r)$, it shares the same upper bound. Furthermore, if $\theta > 0$, there exists $c = c(\theta) > 0$ such that for $r \in (0, 2]$,
\begin{align} \label{bound2}
|g^{(\theta)}_r| \, \leq \, c
\begin{cases}
r^{-(\frac{d}{2} - \theta)} \log(r^{-1} + 1), & \textrm{if } 0 < \theta < \frac{d}{2};\\
\log^2(r^{-1} + 1), & \textrm{if } \theta = \frac{d}{2};\\
1, & \textrm{if } \theta >  \frac{d}{2}.
\end{cases}
\end{align}
\end{lem}

\begin{proof}
We first prove \eqref{bound1}. By the definition of $h^{(\theta)}_r$, it suffices to consider the case $r \in (0, 1]$. The case $\theta = 0$ is exactly the heat kernel bound \eqref{upper pt}. As for the case $\theta > 0$, \eqref{bound1} is a consequence of \eqref{upper pt} together with the following identity:
\beqs
\sum_{i = 1}^{\infty} \lambda_i^{- \theta} \e^{- \lambda_i r} \phi_i^2(x) \, = \, \frac{1}{\Gamma(\theta)} \int_{0}^{\infty} \big( p_{t + r}(x, x) - 1 \big) t^{\theta - 1} \d t.
\eeqs
Notice that
\beqs
g^{(\theta)}_r \, = \, \frac{\Gamma'(\theta)}{\Gamma(\theta)^2} \int_{0}^{\infty} \mu\big( p_{t + r}(\cdot, \cdot) - 1 \big) t^{\theta - 1} \d t - \frac{1}{\Gamma(\theta)} \int_{0}^{\infty} \mu \big( p_{t + r}(\cdot, \cdot) - 1 \big) \ln t \cdot t^{\theta - 1} \d t.
\eeqs
\eqref{bound2} follows by using \eqref{upper pt}.
\end{proof}

We shall also use the Poincar\'e inequality: For any $f \in L^2(\mu)$ with $\mu( f ) = 0$,
\beq \label{Poin}
\| P_t f\|_2 \, \leq \, \e^{-\lambda_1 t} \| f \|_2, \quad t \geq 0.
\eeq
Combining this with \eqref{general sobolev}, we have the $L^p$-Poincar\'e inequality: For any $1 \leq p \leq \infty$, there exists $c > 0$ such that for any $f \in L^p(\mu)$ with $\mu( f ) = 0$,
\beq \label{general Poin}
\| P_t f\|_p \, \leq \, c \e^{-\lambda_1 t} \| f \|_p, \quad t \geq 0.
\eeq

In what follows, to investigate $\mu_t^{(\alpha)}$, we will concern with the following regularized weighted empirical measure
\begin{align*}
\mu_{t, r}^{(\alpha)} \, := \, P_r^* \mu_t^{(\alpha)} \, = \, \frac{\alpha }{ t^\alpha } \int_{0}^{t}  P_r^* \delta_{X_s} s^{\alpha - 1} \d s, \quad  r > 0.
\end{align*}
The density function of $\mu_{t, r}^{(\alpha)}$ w.r.t. $\mu$, denote by $f_{t, r}^{(\alpha)}$, is
\begin{align} \label{def0}
f_{t, r}^{(\alpha)}(y) &\, = \, \mu_t^{(\alpha)}(p_r(\cdot, y)) \, = \, \frac{\alpha}{t^\alpha} \int_{0}^{t} p_r (X_s, y) s^{\alpha - 1} \d s \nonumber\\
& \, = \, 1 + \sum_{i = 1}^{\infty} \e^{-\lambda_i r} \mu_t^{(\alpha)}(\phi_i) \phi(y), \qquad y \in M.
\end{align}

\subsection{Wasserstein distance} \label{wd and sobo}

In this sub-section, we list some lemmas aiming to control the Wasserstein distance. The following result, taken from \cite{AMB, L17}, is essential to the upper bound estimate. For a proof based on the Benamou-Brenier formula, see \cite{AG18}.

\begin{lem} \label{basic wasserstein}
Let $f_0,f_1\in L^2(\mu)$ be two probability density functions w.r.t. $\mu$. Then
\beqs
\W_2^2(f_0\mu, f_1\mu) \, \leq \, 4 \int_M \frac{|\nabla(-\L)^{-1}(f_0 - f_1)|^2}{f_0} \d \mu,
\eeqs
and
\beqs
\W_2^2(f_0\mu, f_1\mu) \, \leq \, \int_M \frac{|\nabla(-\L)^{-1}(f_0 - f_1)|^2}{\M(f_0, f_1)} \d \mu,
\eeqs
where $\M(a, b) : = \frac{a - b}{\ln a - \ln b}$ for $a, b > 0$, and $\M(a, b)  = 0$ if $a$ or $b$ is zero.
\end{lem}
The following lemma is needed when we deal with the quantity $\M(\cdot, 1)^{-1}$.
\begin{lem} \label{fluc of density}
Given $q \ge 1$, there exists $c = c(q) > 0$ such that for any $\varrho \in (0, 1/2)$ and non-negative random variable $Y \in L^2(\Omega)$ defined on a probability space $(\Omega, \mathcal{F}, \PP)$,
\beqs
\E \big[  |\M((1 - \varrho)Y + \varrho, 1)^{-1}-1 |^q     \big]^\frac1q \, \leq \, c \log (\varrho^{-1}) \, \E \big[ |Y - 1|^2 \big]^{\frac{1}{q \vee 2}}.
\eeqs
\end{lem}

\begin{proof}
We start with the case $q \ge 2$. By the definition of $\M(y, 1)$,
\beqs
|\M(y, 1)^{-1}-1 | \, = \, \left| \frac{\ln y}{y - 1} - 1 \right| \, = \, \left| \int_{0}^{1} \frac{u}{1 + (y - 1)u} \d u \right| \cdot |y - 1|, \quad y > 0.
\eeqs
Then let $\eta \in (0, 1)$. If $|y - 1| \leq \eta$, since $1 + (y - 1)u \ge 1 - \eta$, we have
\beq \label{log1}
|\M(y, 1)^{-1}-1 | \, \le \, \frac{1}{2(1 - \eta)}\cdot |y - 1|;
\eeq
on the other hand, if $|y - 1| > \eta$, we have
\beq \label{log2}
|\M(y, 1)^{-1}-1 | \, \le \, 1 + \eta^{-1} \cdot \begin{cases}
\ln (y^{-1}), & \textrm{if } y \in [0, 1 - \eta);\\
\ln (1 + \eta), & \textrm{if } y \in (1+\eta, \infty).
\end{cases}
\eeq
Now we consider the event
$$
A \, = \, \{ \omega \in \Omega : |Y(\omega) - 1| \, \leq \, 1/2  \}.
$$
Since $\varrho \in (0, 1/2)$, it is easily seen that on $A$, $|(1 - \varrho)Y + \varrho - 1| \leq 1/2$; while on $A^c$, $|(1 - \varrho)Y + \varrho - 1| > 1/4$. Then applying \eqref{log1} for $\eta = 1/2$ and \eqref{log2} for $\eta = 1/4$ repectively, we know that there exists $c > 0$ such that
\beqs
\mathbbm{1}_A |\M((1 - \varrho)Y + \varrho, 1)^{-1}-1 |^q \, \leq \, c  \mathbbm{1}_A |Y - 1|^q;
\eeqs
and
\beqs
\mathbbm{1}_{A^c} |\M((1 - \varrho)Y + \varrho, 1)^{-1}-1 |^q \, \leq \, c \mathbbm{1}_{A^c} \log^q(\varrho^{-1}),
\eeqs
where we used $[(1 - \varrho)Y + \varrho]^{-1} \leq \varrho^{-1}$ since $Y \geq 0$.

Consequently, we have
\beqs
\E \big[ \mathbbm{1}_A |\M((1 - \varrho)Y + \varrho, 1)^{-1}-1 |^q   \big] \, \leq \, c \E \big[  \mathbbm{1}_A |Y - 1|^q   \big] \, \leq \,c 2^{-(q-2)} \E \big[  |Y - 1|^2   \big],
\eeqs
and
\beqs
\E \big[ \mathbbm{1}_{A^c} |\M((1 - \varrho)Y + \varrho, 1)^{-1}-1 |^q   \big] \, \leq \, c \log^q(\varrho^{-1}) \PP (A^c ) \, \leq \, 4 c \log^q(\varrho^{-1}) \E \big[  |Y - 1|^2   \big],
\eeqs
where in the last line we used Chebyshev's inequality. Adding the above two inequalities we conclude the proof of the case $q \ge 2$. The case $1 \le q < 2$ is a direct consequence of H\"older's inequality together with the case $q = 2$.
\end{proof}

The following well-known lemma addresses contraction properties in Wasserstein distance along the heat flow (see \cite[Theorem 3]{EKS15}).

\begin{lem} \label{heat wasserstein1}
\begin{enumerate}[(i)]
\item There exists $c > 0$ such that for any $\nu \in \Pr(M)$ and $t > 0$,
\beqs
\W_2^2(P_t^* \nu, \nu) \, \leq \, c t.
\eeqs

\item For any $\nu_0, \nu_1 \in \Pr(M)$,
\beqs
\W_2 (P_t^* \nu_0, P_t^* \nu_1) \, \leq \, \e^{-K t} \W_2 (\nu_0, \nu_1),
\eeqs
where the constant $K \in \R$ comes from the curvature-dimension condition $CD(K, \infty)$.
\end{enumerate}
\end{lem}

\section{Energy estimates} \label{S1}
\setcounter{equation}{0}

Before going further, we need the following elementary result, which plays a role in our proofs. For $\alpha > 0$, the incomplete Gamma function with parameter $\alpha$ is defined by
\begin{align} \label{gamma}
G^{(\alpha)} (u) \, : = \, \int_{u}^{\infty} \e^{-y} y^{\alpha - 1} \d y, \quad u > 0.
\end{align}
Notice that $G^{(\alpha)} (u) = \e^{-u} u^{\alpha - 1}(1 + O(u^{-1}))$ as $u \to \infty$, then a simple computation gives:

\begin{lem} \label{I4}
There exists $c > 0$ such that for $y \geq 2$,
\begin{align*}
 \, \iint_{\{0 \, \leq \, u \, \leq \, v \, \leq \, y  \}} \e^{-(v - u)} u^{\alpha - 1} v^{\alpha - 1} \d u \d v   &\, = \,   \int_{0}^{y}  \e^u u^{\alpha - 1} \big( G^{(\alpha)}(u) - G^{(\alpha)}(y) \big) \d u \\
& \, = \,
\begin{cases}
O(1), & \textrm{if }\alpha \in (0, \frac12);\\
\ln y \, + \, O(1), & \textrm{if }\alpha = \frac12;\\
\frac{y^{2\alpha - 1}}{2\alpha - 1} \, + \, O(1), & \textrm{if }\alpha \in (\frac12, 1);\\
\frac{y^{2\alpha - 1}}{2\alpha - 1} \, + \, O(y^{2\alpha - 2}), & \textrm{if }\alpha \in [1, \infty).
\end{cases}
\end{align*}
and
\begin{align*}
\int_{0}^{y}  \e^u u^{\alpha - 1} G^{(\alpha)}(u) \d u  \, \leq \, c
\begin{cases}
1, & \textrm{if }\alpha \in (0, \frac12);\\
\ln y, & \textrm{if }\alpha = \frac12;\\
y^{2\alpha - 1}, & \textrm{if }\alpha \in (\frac12, \infty).
\end{cases}
\end{align*}
\end{lem}

Thought this paper, we set $t_M = 2(1 \vee \lambda_1^{-1})$, which is a sufficiently large time for our arguments. The following lemma evaluates the $L^2$-norm of $f_{t, r}^{(\alpha)} - 1$, which will be used to govern error terms appearing in our proofs.

\begin{lem}[$L^2$-norm estimate] \label{square bound}
Assume that $M$ is compact with dimension $d \leq 3$. Then there exists $c > 0$ such that for any $t \geq t_M$ and $r \in (0, 1]$,
\begin{enumerate}[(i)]
\item When $d = 1$,
\begin{align*}
\sup_{\nu \in \Pr(M)} \E^\nu \big[ \|f_{t, r}^{(\alpha)} - 1  \|_2^2    \big] \, \leq \,  c
\begin{cases}
t^{-2\alpha} r^{-(\frac12 - 2\alpha)}, & \textrm{if }\alpha \in (0, \frac14);\\
t^{-2\alpha} \log(r^{-1} + 1), & \textrm{if }\alpha = \frac14;\\
t^{-2\alpha}, & \textrm{if }\alpha \in (\frac14, \frac12);\\
t^{-1} \log (t), & \textrm{if }\alpha = \frac12;\\
t^{-1}, & \textrm{if }\alpha \in (\frac12, \infty).
\end{cases}
\end{align*}

\item When $d = 2$,
\beqs
\sup_{\nu \in \Pr(M)} \E^\nu \big[ \|f_{t, r}^{(\alpha)} - 1  \|_2^2    \big] \, \leq \, c \log(r^{-1} + 1) \cdot
\begin{cases}
t^{-2\alpha} r^{-(1 - 2\alpha)}, & \textrm{if }\alpha \in (0, \frac12);\\
t^{-1} \log(\frac{t}{r}) , & \textrm{if }\alpha = \frac12;\\
t^{-1}, & \textrm{if }\alpha \in (\frac12, \infty).
\end{cases}
\eeqs

\item When $d = 3$,
\beqs
\sup_{\nu \in \Pr(M)} \E^\nu \big[ \|f_{t, r}^{(\alpha)} - 1  \|_2^2    \big] \, \leq \, c r^{-\frac12} \cdot
\begin{cases}
t^{-2\alpha} r^{-(1 - 2\alpha)}, & \textrm{if }\alpha \in (0, \frac12);\\
t^{-1} \log(\frac{t}{r}), & \textrm{if }\alpha = \frac12;\\
t^{-1}, & \textrm{if }\alpha \in (\frac12, \infty).
\end{cases}
\eeqs
\end{enumerate}
\end{lem}

\begin{proof}
Recalling that
\beqs
f_{t, r}^{(\alpha)}(y) - 1 \, = \, \sum_{i = 1}^{\infty} \e^{-\lambda_i r} \mu_t^{(\alpha)}(\phi_i) \phi_i(y)
\eeqs
and $\{ \phi_i  \}_{i \geq 1}$ is orthonormal, we have
\beqs
\|f_{t, r}^{(\alpha)} - 1  \|_2^2  \, = \, \sum_{i = 1}^{\infty} \e^{-2 \lambda_i r} |\mu_t^{(\alpha)}(\phi_i)|^2.
\eeqs

The Markov property together with $P_s \phi_i = \e^{-s\lambda_i} \phi_i$ implies that for any $v \ge u \ge 0$,
\beqs
\E^\nu [ \phi_i (X_{v}) | X_t, \, t \leq u ] \, = \, P_{v - u} \phi_i (X_u) \, =  \, \e^{-(v - u) \lambda_i} \phi_i (X_u).
\eeqs
Then using Fubini's theorem, we obtain
\begin{align} \label{L2 norm}
& \E^\nu \big[ \|f_{t, r}^{(\alpha)} - 1  \|_2^2    \big] \, = \, \E^\nu \bigg[ \sum_{i = 1}^{\infty} \e^{-2r\lambda_i} |\mu_t^{(\alpha)}(\phi_i)|^2  \bigg] \nonumber\\
& \, = \, \frac{2 \alpha^2}{t^{2 \alpha}} \sum_{i = 1}^{\infty}  \iint_{\{0 \, \leq \, u \, \leq \, v \, \leq \, t  \}}  \E^\nu \left[  \phi_i^2(X_u) \right]  \e^{- (v - u + 2r) \lambda_i} u^{\alpha - 1} v^{\alpha - 1} \d u \d v \nonumber\\
& \, = \, \frac{2 \alpha^2}{t^{2 \alpha}}  \int_{0}^{t} \int_{u}^{t} \nu \big( P_{u} h_{v - u + 2r}^{(0)}  \big)  u^{\alpha - 1} v^{\alpha - 1} \d v \d u,
\end{align}
where $h_{v - u + 2r}^{(0)}$ is defined by \eqref{theta function}. Then we can write
\beq \label{decomposition}
\E^\nu \big[ \|f_{t, r}^{(\alpha)} - 1  \|_2^2    \big] \, = \, \I_1 \, + \, \I_2(\nu),
\eeq
where
\begin{align*}
\I_1 &\, := \, \frac{2 \alpha^2}{t^{2 \alpha}}  \int_{0}^{t} \int_{u}^{t} \overline{h}_{v - u + 2r}^{(0)} \,   u^{\alpha - 1} v^{\alpha - 1} \d v\d u;\\
 \I_2 (\nu) &\, := \, \frac{2 \alpha^2}{t^{2 \alpha}}  \int_{0}^{t} \int_{u}^{t} \nu \big( P_u h_{v - u + 2r}^{(0)} -  \overline{h}_{v - u + 2r}^{(0)} \big)  u^{\alpha - 1} v^{\alpha - 1} \d v \d u.
\end{align*}

For $\I_1$, we make use of Lemma \ref{I4} to assert that for $t \ge t_M$ and $r > 0$,
\begin{align}
\I_1  & \, = \, \sum_{i = 1}^{\infty} \frac{2 \alpha^2 \e^{-2r\lambda_i}}{t^{2 \alpha} \lambda_i^{2\alpha}} \int_{0}^{t\lambda_i} \e^s  s^{\alpha - 1}  \big( G^{(\alpha)}(s) - G^{(\alpha)}(t \lambda_i) \big) \d s \nonumber\\
& \, \leq \, c \begin{cases}
t^{-2\alpha}  \overline{h}^{(2\alpha)}_{2r}, & \textrm{if }\alpha \in (0, \frac12);\\
\smallskip
t^{-1} \ln t \cdot \overline{h}^{(1)}_{2r} + t^{-1} |g^{(1)}_{2r}| , & \textrm{if }\alpha = \frac12;\\
\smallskip
t^{-1}  \overline{h}^{(1)}_{2r}, & \textrm{if } \alpha \in (\frac12, \infty),
\end{cases}
\end{align}
where $\overline{h}^{(2\alpha)}_{2r}$, $\overline{h}^{(1)}_{2r}$ and $g^{(1)}_{2r}$ are defined by \eqref{theta function'}, whose upper bounds have been estabilished in Lemma \ref{trace}.

Now we turn to the analysis of $\I_2(\nu)$. Since $h_{v - u + 2r}^{(0)} \in C_b(M)$, we have
\beqs
\sup_{\nu \in \Pr(M)} \big| \nu \big( P_{u} h_{v - u + 2r}^{(0)} -  \overline{h}_{v - u + 2r}^{(0)} \big) \big| \, = \, \big\| P_{u} \big(  h_{v - u + 2r}^{(0)} -  \overline{h}_{v - u + 2r}^{(0)}  \big)   \big\|_\infty.
\eeqs
The $L^\infty$-Poincar\'e inequality \eqref{general Poin} implies
\beqs
\big\| P_{u} \big(  h_{v - u + 2r}^{(0)} -  \overline{h}_{v - u + 2r}^{(0)}  \big)   \big\|_\infty \, \leq \, c \e^{- \lambda_1 u} \|  h_{v - u + 2r}^{(0)} \|_\infty.
\eeqs
Combined with Lemma \ref{trace}, it follows that
\beqs
\sup_{\nu \in \Pr(M)} \big| \nu \big( P_{u} h_{v - u + 2r}^{(0)} -  \overline{h}_{v - u + 2r}^{(0)} \big) \big| \, \leq \, c \big( 1 \wedge (v - u + 2r) \big)^{-\frac{d}{2}} \e^{-\lambda_1 (v + 2r)}.
\eeqs
Thus for $r \in (0, 1]$,
\begin{align*}
\sup_{\nu \in \Pr(M)} |\I_2(\nu)| \, \leq \, c t^{-2\alpha}  \int_{0}^{t} \bigg( \int_{0}^{v} \big( 1 + ( v - u + 2r )^{-\frac{d}2} \big)\, u^{\alpha - 1} \d u \bigg) \e^{-\lambda_1 v} v^{\alpha - 1}  \d v,
\end{align*}
Consequently, for $t \geq t_M$ and $r \in (0, 1]$, when $d = 1$,
\begin{align}
\sup_{\nu \in \Pr(M)} |\I_2(\nu)| \, &\leq \, c t^{-2\alpha} \bigg( 1 + r^{-\frac12}\int_{0}^{r}    v^{2\alpha - 1} \d v + \int_{r}^{t} \e^{-\lambda_1 v} v^{2\alpha - \frac32}  \d v     \bigg)\nonumber\\
\, &\leq \,  c t^{-2\alpha} \cdot
\begin{cases}
r^{-(\frac12 - 2\alpha)}, & \textrm{if }\alpha \in (0, \frac14);\\
\log(r^{-1} + 1), & \textrm{if }\alpha = \frac14;\\
1, & \textrm{if }\alpha \in (\frac14, \infty).
\end{cases}
\end{align}
Similarly, when $d = 2$,
\begin{align}
\sup_{\nu \in \Pr(M)} |\I_2(\nu)| \, &\leq \, c t^{-2\alpha} \bigg( 1 + r^{-1} \int_{0}^{r}    v^{2\alpha - 1} \d v + \int_{r}^{t} \log(r^{-1}v) \e^{-\lambda_1 v} v^{2\alpha - 2}  \d v     \bigg)\nonumber\\
\, &\leq \,  c t^{-2\alpha} \log(r^{-1} + 1) \cdot
\begin{cases}
r^{-(1 - 2\alpha)}, & \textrm{if }\alpha \in (0, \frac12);\\
\log(r^{-1} + 1), & \textrm{if }\alpha = \frac12;\\
1, & \textrm{if }\alpha \in (\frac12, \infty).
\end{cases}
\end{align}
And when $d = 3$,
\begin{align} \label{J2'}
\sup_{\nu \in \Pr(M)} |\I_2(\nu)| \, \leq \,  c t^{-2\alpha} r^{-\frac12} \cdot
\begin{cases}
r^{-(1 - 2\alpha)}, & \textrm{if }\alpha \in (0, \frac12);\\
\log(r^{-1} + 1), & \textrm{if }\alpha = \frac12;\\
1, & \textrm{if }\alpha \in (\frac12, \infty).
\end{cases}
\end{align}
Combining \eqref{decomposition}-\eqref{J2'} completes the proof.

\end{proof}

For later developments, we set
\beq \label{notation}
F_{t, r}^{(\alpha)} \, = \, \sup_{\nu \in \Pr(M)} \E^\nu \big[ \|f_{t, r}^{(\alpha)} - 1  \|_2^2    \big], \quad Q_{t, r}^{(\alpha)} \, = \, R_\alpha (t)^{-1} \int_{0}^{r} F_{t, s}^{(\alpha)} \, \d s.
\eeq
Then Lemma \ref{square bound} implies the following statement.
\begin{cor} \label{square bound'}
If we choose $r = r(t) = t^{-\beta}$ for some $\beta > 0$, then there exists $c = c(\beta) > 0$ such that for any $t \geq t_M$,
\begin{enumerate}[(i)]
\item When $d = 1$,
\begin{equation*}
F_{t, r}^{(\alpha)} \, \leq \,  c R_\alpha (t) \cdot
\begin{cases}
t^{(\frac12 - 2\alpha)\beta}, & \textrm{if }\alpha \in (0, \frac14);\\
\log(t), & \textrm{if }\alpha = \frac14;\\
1, & \textrm{if }\alpha \in (\frac14, \infty).
\end{cases}\\
Q_{t, r}^{(\alpha)} \, \leq \, c
\begin{cases}
t^{-(2\alpha + \frac12)\beta}, & \textrm{if }\alpha \in (0, \frac14);\\
t^{-\beta} \log(t), & \textrm{if }\alpha = \frac14;\\
t^{-\beta}, & \textrm{if }\alpha \in (\frac14, \infty).
\end{cases}
\end{equation*}

\item When $d = 2$,
\begin{equation*}
F_{t, r}^{(\alpha)} \, \leq \,  c R_\alpha (t)\log(t) \cdot
\begin{cases}
t^{(1 - 2\alpha)\beta}, & \textrm{if }\alpha \in (0, \frac12);\\
1, & \textrm{if }\alpha \in [\frac12, \infty).
\end{cases}\\
Q_{t, r}^{(\alpha)} \, \leq \,  c \log(t) \cdot
\begin{cases}
t^{- 2\alpha \beta}, & \textrm{if }\alpha \in (0, \frac12);\\
t^{-\beta}, & \textrm{if }\alpha \in [\frac12, \infty).
\end{cases}
\end{equation*}

\item When $d = 3$,
\beqs
F_{t, r}^{(\alpha)} \, \leq \, c R_\alpha (t) \cdot
\begin{cases}
t^{(\frac32 - 2\alpha)\beta}, & \textrm{if }\alpha \in (0, \frac12);\\
t^{\frac{\beta}2}, & \textrm{if }\alpha \in [\frac12, \infty),
\end{cases}\\
Q_{t, r}^{(\alpha)} \, \leq \, c
\begin{cases}
t^{- (2\alpha - \frac12) \beta}, & \textrm{if }\alpha \in (\frac14, \frac12);\\
t^{-\frac{\beta}{2}}, & \textrm{if }\alpha \in [\frac12, \infty).
\end{cases}
\eeqs

\end{enumerate}
\end{cor}

The following propositions are devoted to asymptotic estimates of the $H^{-1, 2}$-Sobolev norm of $f_{t, r}^{(\alpha)} - 1$, which contributes the main term of our results.

\begin{propo}[$H^{-1, 2}$-Sobolev norm estimate for $\alpha \ge \frac12$] \label{energy}
Assume that $M$ is compact with dimension $d \leq 3$. For $\alpha \in [\frac12, \infty)$, there exists $c > 0$ such that for any $t \geq t_M$ and $r \in (0, 1]$,
\begin{align*}
 \sup_{\nu \in \Pr(M)} & \bigg| R_\alpha(t)^{-1} \E^\nu \left[ \int_M |\nabla (-\L)^{-1} (f_{t, r}^{(\alpha)} - 1)|^2 \d\mu \right] -  c(\alpha) \sum_{i = 1}^{\infty} \lambda_i^{-2} \bigg|\\
\quad \,& \leq \,e(r) \, + \, c
\begin{cases}
\log(t)^{-1}, & \textrm{if }\alpha = \frac12;\\
t^{-2(\alpha \wedge 1) + 1}, & \textrm{if }\alpha \in (\frac12, \infty),
\end{cases}
\end{align*}
where
\beqs
c(\alpha) \, = \,
\begin{cases}
\frac12, & \textrm{if }\alpha = \frac12;\\
\frac{2 \alpha^2}{2\alpha - 1} , & \textrm{if }\alpha \in (\frac12, \infty),
\end{cases}\\\quad
e(r) \, = \, c
\begin{cases}
r, & \textrm{when }d = 1;\\
r\log(r^{-1} + 1), & \textrm{when }d = 2;\\
r^{\frac12} , & \textrm{when }d = 3.
\end{cases}
\eeqs
\end{propo}

\begin{proof} We only deal with the case $\alpha > \frac12$, the other case is similar. Recalling that $(-\L)^{-1} = \int_{0}^{\infty} P_s \d s$, then integration by parts gives
\beqs
\int_M |\nabla (-\L)^{-1} (f_{t, r}^{(\alpha)} - 1)|^2 \d \mu  \, = \, \int_{0}^{\infty} \| P_{\frac{s}{2}} f_{t, r}^{(\alpha)} - 1 \|_2^2 \d s.
\eeqs
Thus by the definition of $f_{t, r}^{(\alpha)}$ and the fact that $\{ \phi_i \}_{i \geq 1}$ is orthonormal,
\beqs
\int_M |\nabla (-\L)^{-1} (f_{t, r}^{(\alpha)} - 1)|^2 \d \mu \, = \, \sum_{i = 1}^{\infty} \frac{|\mu_t^{(\alpha)}(\phi_i)|^2}{\lambda_i \e^{2r\lambda_i}}.
\eeqs
By the same calculation as in Lemma \ref{square bound}, we obtain
\begin{align} \label{sobolev norm2}
& \E^\nu \left[ \int_M |\nabla (-\L)^{-1} (f_{t, r}^{(\alpha)} - 1)|^2 \d \mu  \right] \, = \, \E^\nu \bigg[ \sum_{i = 1}^{\infty} \frac{|\mu_t^{(\alpha)}(\phi_i)|^2}{\lambda_i \e^{2r\lambda_i}}  \bigg] \nonumber\\
& \, = \, \frac{2 \alpha^2}{t^{2\alpha }} \sum_{i = 1}^{\infty} \frac{1}{\lambda_i}  \iint_{\{0 \, \leq \, u \, \leq \, v \, \leq \, t  \}} \E^\nu \left[  \phi_i^2(X_{u}) \right] \e^{- (v - u + 2r) \lambda_i}  u^{\alpha - 1} v^{\alpha - 1} \d u \d v \nonumber\\
& \, = \, \frac{2 \alpha^2}{t^{2\alpha }} \int_{0}^{t} \int_{u}^{t} \nu \big( P_{u} h_{v - u + 2r}^{(1)}  \big)  u^{\alpha - 1} v^{\alpha - 1} \d v \d u,
\end{align}
where $h_{v - u + 2r}^{(1)}$ is defined by \eqref{theta function}. In analogy with \eqref{decomposition}, we write
\begin{align}\label{sobolev norm}
t \, \E^\nu \left[ \int_M |\nabla (-\L)^{-1} (f_{t, r}^{(\alpha)} - 1)|^2 \d \mu  \right] \, = \, \J_1 \, + \, \J_2(\nu),
\end{align}
where
\begin{align*}
\J_1 & \, := \, \frac{2 \alpha^2}{t^{2\alpha - 1 }}  \int_{0}^{t} \int_{u}^{t} \overline{h}_{v - u + 2r}^{(1)} \, u^{\alpha - 1} v^{\alpha - 1} \d v \d u;\\
\J_2(\nu) & \, := \, \frac{2 \alpha^2}{t^{2\alpha -1}}  \int_{0}^{t} \int_{u}^{t} \nu \big( P_{u} h_{v - u + 2r}^{(1)} -  \overline{h}_{v - u + 2r}^{(1)} \big) u^{\alpha - 1} v^{\alpha - 1} \d v \d u.
\end{align*}

By Lemma \ref{I4}, for $t \geq t_M$ and $r \in (0, 1]$,
\begin{align} \label{I1 bound}
\J_1 \, &= \, \sum_{i = 1}^{\infty} \frac{2\alpha^2 \e^{-2r\lambda_i}}{t^{2\alpha - 1}\lambda_i^{2\alpha + 1} } \int_{0}^{t\lambda_i} \e^u  s^{\alpha - 1}  \big( G^{(\alpha)}(s) - G^{(\alpha)}(t \lambda_i) \big) \d s \nonumber\\
 \, &= \, \frac{2\alpha^2}{2\alpha - 1} \sum_{i = 1}^{\infty} \lambda_i^{-2} \e^{-2r\lambda_i} \, + \,
\begin{cases}
t^{-(2\alpha - 1)} \cdot O \big( \overline{h}_{2r}^{(2\alpha + 1)} \big), & \textrm{if }\alpha \in (\frac12, 1);\\
t^{-1} \cdot  O \big( \overline{h}_{2r}^{(3)} \big), & \textrm{if }\alpha \in [1, \infty).
\end{cases}
\end{align}

On the other hand, by a similar argument in Lemma \ref{square bound},
\begin{align*}
\big| \nu \big( P_{u} h_{v - u + 2r}^{(1)} -  \overline{h}_{v - u + 2r}^{(1)} \big) \big|  \, \leq \, \big\| P_{u} \big( h_{v - u + 2r}^{(1)} -  \overline{h}_{v - u + 2r}^{(1)} \big) \big\|_\infty \, \leq \, c \e^{-\lambda_1 u} \| h_{v - u + 2r}^{(1)} \|_\infty.
\end{align*}
Therefore, for $t \geq t_M$ and $r \in (0, 1]$,
\begin{align} \label{I2 bound}
\sup_{\nu \in \Pr(M)} |\J_2(\nu)| \, &\leq \, c t^{-(2\alpha - 1)}  \int_{0}^{t} \int_{u}^{t} \e^{-\lambda_1 u} \| h_{v - u + 2r}^{(1)} \|_\infty \, u^{\alpha - 1} v^{\alpha - 1} \d v \d u\nonumber\\
\, &\leq \, c t^{-(2\alpha - 1)}  \int_{0}^{t} \bigg( \int_{0}^{v} \big( 1 + (v - u + 2r)^{-\frac12} \big) u^{\alpha - 1} \d u  \bigg) \e^{-\lambda_1 v} v^{\alpha - 1}  \d v \le \, c t^{-(2\alpha - 1)},
\end{align}
where in the last line we used Lemma \ref{trace} and the fact that $d \le 3$.

Notice that
\beq \label{error bound}
\sum_{i = 1}^{\infty} \lambda_i^{-2} - \sum_{i = 1}^{\infty} \lambda_i^{-2} \e^{-2r\lambda_i} \, = \, \int_{0}^{2r} \overline{h}^{(1)}_s \d s.
\eeq
Combining \eqref{sobolev norm}-\eqref{error bound}, the proof is then complete after the use of Lemma \ref{trace}.
\end{proof}

\begin{remark} \label{remark0}
The proof of Proposition \ref{energy} indeed gives the following upper bound for the $H^{-1, 2}$-Sobolev norm of $f_{t, r}^{(\alpha)} - 1$, which keeps valid for any dimension $d$ and $\alpha > 0$. By combining \eqref{sobolev norm}-\eqref{I2 bound} and Lemma \ref{I4}, we have
\begin{align*}
\E^\nu \left[ \int_M |\nabla (-\L)^{-1} (f_{t, r}^{(\alpha)} - 1)|^2 \d \mu  \right] \, &\le \, c t^{-2\alpha}  \iint_{\{0 \, \leq \, u \, \leq \, v \, \leq \, t  \}}  \e^{-\lambda_1 u} \| h_{v - u + 2r}^{(1)} \|_\infty \, u^{\alpha - 1} v^{\alpha - 1} \d u   \d v\\
\, & \qquad + \, c
\begin{cases}
t^{-2\alpha} \overline{h}^{(2\alpha + 1)}_{2r}, & \textrm{if } \alpha \in (0, \frac12);\\
t^{-1} \big( \ln t \cdot \overline{h}^{(2)}_{2r}  \, + \, |g^{(2)}_{2r}| \big), & \textrm{if } \alpha = \frac12;\\
t^{-1}\overline{h}^{(2)}_{2r}, & \textrm{if } \alpha \in (\frac12, \infty),
\end{cases}
\end{align*}
where $t \geq t_M$ and $r \in (0, 1]$.

\end{remark}

Now we turn to the case $0 < \alpha < \frac12$. We shall make use of the connection between $L^2$-norm and $H^{-1, 2}$-Sobolev norm of $f_{t, r}^{(\alpha)} - 1$. To this end, we first define the following functions: Given $\alpha \in (\frac{(d - 2)^+}{4}, \frac12)$, for $t \geq t_M$ and $r \in (0, 1]$,
\begin{align} \label{defofF}
 H_{t, r}^{(\alpha)}(x) &\,: = \, 2 \alpha^2  \iint_{\{0 \, \leq \, u \, \leq \, v \, \leq \, t  \}} \big(  P_{u} h_{v - u + 2r}^{(1)} \big) (x) \, u^{\alpha - 1} v^{\alpha - 1} \d u \d v; \nonumber\\
 H_t^{(\alpha)} (x) &\, : = \, 2 \alpha^2  \iint_{\{0 \, \leq \, u \, \leq \, v \, \leq \, t  \}} \big( P_{u} h_{v - u}^{(1)} \big) (x) \, u^{\alpha - 1} v^{\alpha - 1} \d u \d v; \nonumber\\
 H^{(\alpha)} (x) &\, : = \, 2 \alpha^2 \iint_{\{0 \, \leq \, u \, \leq \, v \, < \, \infty  \}} \big(  P_{u} h_{v - u}^{(1)} \big) (x) \, u^{\alpha - 1} v^{\alpha - 1} \d u \d v, \quad x \in M.
\end{align}
It can be observed that these functions have the following properties:
\begin{enumerate}[(i)]
\item $H_{t, r}^{(\alpha)}$, $H_t^{(\alpha)}$ and $H^{(\alpha)}$ are all bounded continuous functions on $M$;

\item $\inf_{x \in M} H^{(\alpha)}(x) > 0$.
\end{enumerate}
Indeed, by the estimates \eqref{part1} and \eqref{part2} given in Proposition \ref{energy'} below, (i) is a consequence of the continuity of $H_{t, r}^{(\alpha)}$ (note that $M$ is compact). To control $\nabla H_{t, r}^{(\alpha)}$, we need the following upper bound on the gradient of $p_t$: There exists $c > 0$ such that for all $t > 0$ and $x, y \in M$,
\beqs
\big| \nabla_x p_t(x, y) \big| \, \leq \, c \e^{-\lambda_1 t} (1 \wedge t)^{-\frac{d + 1}{2}}.
\eeqs
This bound is a consequence of the local Poincar\'e inequality and the heat kernel bound \eqref{upper pt}. Then recalling the definition of $h_{s}^{(1)}$ and using the chain rule, for $s > 0$, we have
\begin{align*}
\| \nabla h_{s}^{(1)} \|_\infty  \, \le \, 2 \int_{s}^{\infty} \sup_{x, y \in M} \big|\nabla_x p_u (x, y) \big| \d u
 \, \le \, c  \e^{-\lambda_1 s} \cdot
\begin{cases}
\log((1 \wedge s)^{-1} + 1), & \textrm{when } d = 1;\\
(1 \wedge s)^{-\frac{d-1}{2}}, & \textrm{when } d \ge 2.
\end{cases}
\end{align*}
Combining this with the gradient bound (see e.g. \cite[Theorem 3.2.4]{BGL14})
$$|\nabla P_u h_{s}^{(1)}| \leq \e^{-K u} P_u (|\nabla h_{s}^{(1)}|), \qquad u > 0,$$
it follows that $|\nabla P_u h_{s}^{(1)}|$ is bounded from above uniformly in $(u, s) \in [0, t] \times [2r, 2r+t]$,
which implies that $\| \nabla H_{t, r}^{(\alpha)} \|_\infty < \infty$. Since $M$ is connected, the continuity of $H_{t, r}^{(\alpha)}$ is obtained. As for the property (ii), by \eqref{general Poin} and Lemma \ref{trace}, there exists $C >0$ such that for any $u > 0$ and $s \in [1, 2]$,
\beqs
\big\| P_u h_s^{(1)} -  \overline{h}^{(1)}_s \big\|_\infty \, \leq \, C \e^{-\lambda_1 u}.
\eeqs
In particular, there exists $T > 0$ such that for any $u \ge T$ and $s \in [1, 2]$,
\beqs
P_u h_s^{(1)} \, \geq \, \frac{1}{2}\overline{h}^{(1)}_s,
\eeqs
so that
\beqs
\inf_{x \in M} H^{(\alpha)} (x) \, \geq \,  \alpha^2 \int_{T}^{\infty} u^{\alpha - 1} (u + 2)^{\alpha - 1} \d u \cdot \int_{1}^{2} \overline{h}^{(1)}_s   \d s  \, > \, 0.
\eeqs

In analogy with Proposition \ref{energy}, the following statement holds.
\begin{propo}[$H^{-1, 2}$-Sobolev norm estimate for $0 < \alpha < \frac12$] \label{energy'}
Assume that $M$ is compact with dimension $d \leq 3$. For $\alpha \in (\frac{(d - 2)^+}{4}, \frac12)$, there exists $c > 0$ such that for any $t \geq t_M$ and $r \in (0, 1]$,
\begin{align*}
 \sup_{\nu \in \Pr(M)} \bigg| R_\alpha(t)^{-1} \E^\nu \left[ \int_M |\nabla (-\L)^{-1} (f_{t, r}^{(\alpha)} - 1)|^2 \d\mu \right] -  \nu(H^{(\alpha)}) \bigg| \, \leq \, c \big( t^{-(1 - 2\alpha)} \,  +  \, Q_{t, r}^{(\alpha)} \big) ,
\end{align*}
where $Q_{t, r}^{(\alpha)}$ is defined by \eqref{notation} and satisfies
\begin{enumerate}[(i)]
\item when $d = 1$,
\begin{equation*}
Q_{t, r}^{(\alpha)} \, \le \, c
\begin{cases}
r^{2\alpha + \frac12}, & \textrm{if } \alpha \in (0, \frac14);\\
r \log(r^{-1} + 1), & \textrm{if } \alpha = \frac14;\\
r, & \textrm{if } \alpha \in (\frac14, \frac12);
\end{cases}
\end{equation*}

\item when $d = 2$,
\begin{equation*}
Q_{t, r}^{(\alpha)} \, \le \, c r^{2\alpha} \log(r^{-1} + 1);
\end{equation*}

\item when $d = 3$,
\begin{equation*}
Q_{t, r}^{(\alpha)} \, \le \, c r^{2\alpha - \frac12}.
\end{equation*}
\end{enumerate}

\end{propo}

\begin{proof}
By \eqref{sobolev norm2} and Fubini's Theorem, we have
\beq \label{part0}
\nu ( H_{t, r}^{(\alpha)} ) \, = \, t^{2\alpha}  \E^\nu \left[  \int_M |\nabla (-\L)^{-1} (f_{t, r}^{(\alpha)} - 1)|^2 \d \mu  \right].
\eeq

Next we will evaluate $\big\|H_{t, r}^{(\alpha)} - H_t^{(\alpha)} \big\|_\infty$ and $\big\| H_t^{(\alpha)} - H^{(\alpha)} \big\|_\infty$. Combining \eqref{L2 norm}, \eqref{sobolev norm2} and \eqref{part0}, for fixed $t \ge t_M$ and $x \in M$,
\begin{align} \label{derivative}
\frac{\d}{\d s}\big( H_{t, s}^{(\alpha)}(x) \big)  \, = \, \frac{\d}{\d s} \bigg\{  t^{2\alpha} \, \E^x \left[ \int_M |\nabla (-\L)^{-1} (f_{t, s}^{(\alpha)} - 1)|^2 \d \mu  \right]    \bigg\}  \, = \, - 2 t^{2\alpha} \, \E^x \left[ \int_M |f_{t, s}^{(\alpha)} - 1|^2 \d \mu   \right].
\end{align}
By the dominated convergence Theorem, $\lim_{r\downarrow 0} H_{t, r}^{(\alpha)} = H_t^{(\alpha)}$. Therefore,
\begin{align} \label{part1}
\big\|H_{t, r}^{(\alpha)} - H_t^{(\alpha)} \big\|_\infty  \, \leq \, \int_{0}^{r} \sup_{x \in M} \left|  \frac{\d}{\d s} H_{t, s}^{(\alpha)}(x)  \right| \d s \, \leq \, 2 Q_{t, r}^{(\alpha)},
\end{align}
where $Q_{t, r}^{(\alpha)}$ is defined by \eqref{notation}.

On the other hand, we have
\begin{align*}
\left| H^{(\alpha)} (x) - H_t^{(\alpha)}(x) \right| \, &= \, \int_{0}^{\infty} \bigg( \int_{t \vee u}^{\infty} \big( P_{u} h_{v - u}^{(1)} \big)(x) \, v^{\alpha - 1} \d v \bigg) u^{\alpha - 1} \d u\\
 \,  &\le \, \int_{0}^{\infty} \bigg( \int_{t \vee u}^{\infty} \| h_{v - u}^{(1)} \|_\infty \, v^{\alpha - 1} \d v \bigg) u^{\alpha - 1} \d u.
\end{align*}
Since $d \le 3$ and $\alpha \in (0, \frac12)$, by Lemma \ref{trace},
\begin{align} \label{part2}
\big\| H^{(\alpha)}  - H_t^{(\alpha)}  \big\|_\infty \, \leq \, c  \int_{0}^{\infty} (t \vee u)^{\alpha - 1} u^{\alpha - 1} \d u \, \leq \, c t^{-(1 - 2\alpha)}.
\end{align}
Combining \eqref{part1} and \eqref{part2}, the proof of this proposition is complete after the use of Lemma \ref{square bound}.
\end{proof}

\begin{remark} \label{remark1}
Similar to \eqref{derivative}, it is worth making the following observation. For any $0 < r_* < r \le 1$,
\begin{align*}
& \int_M |\nabla (-\L)^{-1} (f_{t, r}^{(\alpha)} - f_{t, r_*}^{(\alpha)})|^2 \d \mu \, = \, \sum_{i = 1}^{\infty} \frac{|\mu_t^{(\alpha)}(\phi_i)|^2}{\lambda_i} \big( \e^{-r_*\lambda_i} - \e^{-r\lambda_i} \big)^2 \\
\quad \, & \le \, \sum_{i = 1}^{\infty} \frac{|\mu_t^{(\alpha)}(\phi_i)|^2}{\lambda_i} \big( \e^{-2r_*\lambda_i} - \e^{-2r\lambda_i} \big) \, = \,
2 \int_{r_*}^{r} \| f_{t, s}^{(\alpha)} - 1 \|^2_2  \d s.
\end{align*}
\end{remark}

By Proposition \ref{energy} and \ref{energy'}, we yield the following conclusion.

\begin{cor} \label{final}
Under the assumptions of Theorem \ref{main}. Then
\begin{align*}
\lim_{(r, \, t) \to (0, \, \infty)} R_\alpha(t)^{-1} \, \E^\nu \left[ \int_M |\nabla (-\L)^{-1} (f_{t, r}^{(\alpha)} - 1)|^2 \d\mu \right] \, = \,
\begin{cases}
\nu ( H^{(\alpha)} ), & \textrm{if } \alpha \in (\frac{(d - 2)^+}{4}, \frac12);\\
\frac12 \sum_{i = 1}^{\infty} \lambda_i^{-2}, & \textrm{if } \alpha = \frac12;\\
\frac{2 \alpha^2}{2\alpha - 1} \sum_{i = 1}^{\infty} \lambda_i^{-2}, & \textrm{if } \alpha \in (\frac12, \infty),
\end{cases}
\end{align*}
where $H^{(\alpha)}$ is defined by \eqref{defofF}. Moreover, above limits converge uniformly in $\nu \in \Pr(M)$.
\end{cor}

\section{Concentration and density fluctuation bounds} \label{s3}
\setcounter{equation}{0}

In this section, we prove a weak form of Bernstein-type inequality for $\mu_t^{(\alpha)}$ and apply it to establish the fluctuation bounds. We begin with the following lemma, which refines \cite[Lemma 2.5]{WZ19}.

\begin{lem} \label{fluctuation1}
For $\alpha > 0$, there exists $c > 0$ such that for any $1 \le p < \infty$, $t \geq t_M$ and bounded measurable function $f : M \to \R$ satisfying $\mu(f) = 0$,
\beqs
\sup_{\nu \in \Pr(M)} \E^\nu \big[ |\mu_t^{(\alpha)}(f) |^p \big]^{\frac{1}{p}} \, \leq c \sqrt{p} \cdot  R_\alpha(t)^{\frac12} \| f \|_\infty.
\eeqs

\end{lem}

\begin{proof} By H\"older's inequality, we only need to consider the case $p = 2k$, $k \in \N^+$. For any $k \in \N^+$, we set
$$
I_{2k}^{(\alpha)}(t) \, := \, t^{2\alpha k} \, \E^\nu \big[ |\mu_t^{(\alpha)}(f) |^{2k} \big] \, = \, \E^\nu  \left[ \left| \int_{0}^{t} f(X_s) s^{\alpha - 1} \d s  \right|^{2k} \right].
$$
Then it is easily seen that
\beq \label{I1}
I_{2k}^{(\alpha)}(t) \, = \,  (2k)! \, \E^\nu \left[ \int_{\Delta_k(t)} f(X_{s_1})\cdots f(X_{s_{2k}})\,  s_1^{\alpha -1} \d s_1 \cdots  \, s_{2k}^{\alpha -1}\d s_{2k} \right],
\eeq
where $\Delta_k (t)$ denotes the simplex $\big\{(s_1, \ldots, s_{2k}) \in [0, t]^{2k}; 0 \le s_1 \le \cdots \le s_{2k} \le t \big\}$. By the Markov property,
$$
\E^\nu \big[f(X_{s_{2k}}) | \, X_t, \, t \le s_{2k-1}    \big]  \, = \, (P_{s_{2k} - s_{2k-1}} f)(X_{s_{2k-1}}).
$$
For any $v \ge u \ge 0$, we set $g(u, v) = (fP_{v - u}f)(X_{u})$. We therefore obtain
\begin{align*}
I_{2k}^{(\alpha)}(t) & \, = \, (2k)! \, \E^\nu \bigg[ \int_{0}^{t} f(X_{s_1}) s_1^{\alpha - 1} \d s_1 \int_{s_1}^{t} f(X_{s_2}) s_2^{\alpha - 1} \d s_2 \cdots \int_{s_{2k-3}}^{t} \! f(X_{s_{2k-2}}) s_{2k-2}^{\alpha - 1} \d s_{2k-2}\\
 & \qquad  \quad \, \cdot \, \int_{s_{2k-2}}^{t} \!  s_{2k-1}^{\alpha - 1} \d s_{2k-1} \! \int_{s_{2k-1}}^t \! g(s_{2k-1}, s_{2k}) s_{2k}^{\alpha - 1} \d s_{2k}   \bigg].
\end{align*}
By Fubini's theorem, $I_{2k}^{(\alpha)}(t)$ may be rewritten as
\begin{align*}
I_{2k}^{(\alpha)}(t) & \,= \,  (2k)! \, \E^\nu \bigg[ \int_{\Delta_1(t)}  \bigg( \int_{\Delta_{k-1}(u)} f(X_{s_1})\cdots f(X_{s_{2k-2}}) s_1^{\alpha -1 } \d s_1 \cdot \cdots s_{2k-2}^{\alpha -1 }\d s_{2k-2}  \bigg) \\
& \qquad \quad \, \cdot \, g(u, v)  u^{\alpha-1} v^{\alpha-1} \d u \d v   \bigg] \\
& \, = \,  \frac{(2k)!}{(2k-2)!} \int_{\Delta_1(t)} \E^\nu \bigg[ g(u, v) \left|\int_0^{u} f(X_s) s^{\alpha - 1}\d s \right|^{2k-2}   \bigg] u^{\alpha-1} v^{\alpha-1}  \d u \d v.
\end{align*}
By H\"older's inequality,
\begin{align} \label{I2}
I_{2k}^{(\alpha)}(t)  &\, \le \,  2k (2k - 1) \int_{\Delta_1(t)}  \E^\nu \big[ |g(u, v)|^k \big]^{\frac{1}{k}} \cdot \E^\nu \bigg[ \left|\int_0^{u} f(X_s)\d s \right|^{2k} \bigg]^{\frac{k-1}{k}} u^{\alpha-1} v^{\alpha-1}\d u \d v \nonumber\\
 & \, = \,  2k (2k - 1) \int_{\Delta_1(t)} \E^\nu \big[ |g(u, v)|^k \big]^{\frac{1}{k}} I_{2k}^{(\alpha)}(u)^{\frac{k-1}{k}} u^{\alpha-1} v^{\alpha-1} \d u \d v.
\end{align}
Notice that $\mu(f) = 0$, by the $L^\infty$-Poincar\'e inequality \eqref{general Poin}, we have
\beqs
|g(u, v)| \, \leq \, \| f \|_\infty \cdot \| P_{v - u} f \|_\infty \, \leq \,  c \e^{-\lambda_1(v - u)} \| f \|_\infty^2.
\eeqs
Combined with \eqref{I2}, it follows that
\begin{align*}
I_{2k}^{(\alpha)}(t)  \leq   c k^2 \lambda_1^{-\alpha} \| f \|_\infty^2 \int_{0}^{t} I_{2k}^{(\alpha)}(u)^{\frac{k-1}{k}} \cdot \e^{\lambda_1 u} u^{\alpha - 1}  G^{(\alpha)}(\lambda_1 u) \d u,
\end{align*}
where $G^{(\alpha)}$ is defined by \eqref{gamma}. Then by the generalized Gronwall's inequality (see \cite{B56}),
\begin{align*}
I_{2k}^{(\alpha)}(t) & \, \leq \, \bigg(  c k  \lambda_1^{-\alpha} \| f \|_\infty^2 \int_{0}^{t} \e^{\lambda_1 u} u^{\alpha - 1}  G^{(\alpha)}( \lambda_1 u)   \d u \bigg)^k\nonumber\\
& \, \le \, \bigg(  c k  \lambda_1^{-2\alpha} \| f \|_\infty^2 \int_{0}^{t \lambda_1} \e^s s^{\alpha - 1}  G^{(\alpha)}(s)  \d s \bigg)^k.
\end{align*}
Using Lemma \ref{I4}, we conclude the proof of Lemma \ref{fluctuation1}.
\end{proof}

\begin{remark}
\begin{enumerate}[(1)]
\item In the case $\alpha = 1$, the central limit theorem for Markov processes tells us that (see e.g. \cite{KV84})
\beqs
\sqrt{t} \mu_t(f) \, \to \, \mathrm{N}(0, \sigma^2(f)) \ \text{in\ law\ as\ }t\to\infty,
\eeqs
where
\beqs
\sigma^2(f) \, : = \, 4 \int_{0}^{\infty} \| P_s f \|_2^2 \d s \, \le \, 2 \lambda_1^{-1} \| f \|_2^2 \, \le \, 2 \lambda_1^{-1} \| f \|_\infty^2,
\eeqs
and $\mathrm{N}(0, \sigma^2(f))$ is the centered Gaussian distribution with variance $\sigma^2(f)$. By the feature of Gaussian distribution, this lemma is sharp in respect to the powers of both $p$ and $t$.

\item For more on Bernstein-type inequalities for symmetric Markov processes, we refer the reader to \cite{L98, GGW14}.
\end{enumerate}
\end{remark}

It is well-known that the moment estimates in Lemma \ref{fluctuation1} imply a probabilistic tail estimate of exponential decay (see e.g. \cite[Exercise 2.3.8]{T21}). Then the following concentration inequality is achieved.
\begin{cor}[Concentration inequality for $\mu_t^{(\alpha)}$] \label{tail}
For $\alpha > 0$, there exists $c > 0$ such that for any $\eta > 0$, $t\geq t_M$ and bounded measurable function $f : M \to \R$ satisfying $\mu(f) = 0$,
\beqs
\sup_{\nu \in \Pr(M)} \PP^\nu \big( |\mu_t^{(\alpha)}(f) | > \eta \big) \, \leq \, c \exp \left( - \frac{ \eta^2}{c R_\alpha(t) \| f \|_\infty^2}   \right)
\eeqs
\end{cor}

If we take $f(x) = p_r(x, y) - 1$, recalling \eqref{def0}, Corollary \ref{tail} quantifies the size of the event $\{ |f_{t, r}^{(\alpha)}(y) - 1| > \eta  \}$. Since $M$ is connected, for $y \ne z$, the events $\{ |f_{t, r}^{(\alpha)}(y) - 1| > \eta  \}$ and $\{ |f_{t, r}^{(\alpha)}(z) - 1| > \eta  \}$ are not independent. More precisely, if we take $f(x) = p_r(x, y) - p_r(x, z)$ in Corollary \ref{tail}, then for $\eta > 0$ and $t\geq t_M$,
\beq \label{bound0}
\sup_{\nu \in \Pr(M)} \PP^\nu \big( |f_{t, r}^{(\alpha)}(y) - f_{t, r}^{(\alpha)}(z)| > \eta \big) \, \leq \, c \exp \left( - \frac{\eta^2}{c R_\alpha(t) \|  p_r(\cdot, y) - p_r(\cdot, z)  \|_\infty^2}   \right).
\eeq
By the local Poincar\'e inequality, there exists $c > 0$ such that for any $x \in M$ and $r \in (0, 1]$,
\beqs
\Lip(p_r(x, \cdot)) \, \leq \, c r^{-\frac{d + 1}{2}}.
\eeqs
which implies that
\beq \label{Bound1}
\|  p_r(\cdot, y) - p_r(\cdot, z)  \|_\infty \, \le \, c r^{-\frac{d + 1}{2}} \rho(y, z).
\eeq
On the other hand, by the heat kernel bound \eqref{upper pt},
\beq \label{Bound2}
\|  p_r(\cdot, y) - p_r(\cdot, z)  \|_\infty \, \le \,  c r^{-\frac{d}{2}}.
\eeq
Now, for each $u > 0$, we define the truncated geodesic distance on $M$ at the level $u$, denoted by $\rho_u$, as follows:
\beq \label{rhou}
\rho_u(x, y) \, = \, \rho(x, y) \wedge u, \quad x, y\in M.
\eeq
Then \eqref{Bound1} together with \eqref{Bound2} yields that, for $r \in (0, 1]$,
\beqs
\|  p_r(\cdot, y) - p_r(\cdot, z)  \|_\infty \, \le \,  c r^{-\frac{d + 1}{2}}  \rho_{r^{\frac12}}(y, z).
\eeqs
Combining this bound with \eqref{bound0}, we obtain
\beq \label{Lip}
\sup_{\nu \in \Pr(M)} \PP^\nu \big( |f_{t, r}^{(\alpha)}(y) - f_{t, r}^{(\alpha)}(z)| > \eta \big) \, \leq \, c \exp \left( - \frac{r^{d + 1} \eta^2}{c R_\alpha(t) \rho_{r^{\frac12}}(y, z)^2}   \right).
\eeq

Therefore, for fixed $t \ge t_M$ and $r \in (0, 1]$, we can view $(f_{t, r}^{(\alpha)}(y))_{y \in M}$ as a sub-Gaussian process indexed by $M$ and then apply the Dudley's entropy integral bound to control its size. The proof of the following lemma could be found in \cite[Chapter 11]{LT91} or \cite[Chapter 2]{T21}.
\begin{lem}[Dudley's entropy integral bound] \label{generic chaining}
Let $(T, \delta)$ be a metric space. For each $\epsilon > 0$, let $N(T, \delta, \epsilon)$ be the smallest number of $\epsilon$-closed balls needed to cover $T$. Let $(Z_t)_{t \in T}$ be a stochastic process indexed by $T$. Assume that $(Z_t)_{t \in T}$ satisfies the sub-Gaussian inequality: There exist $A, B > 0$ such that for any $s, t \in T$ and $\eta > 0$,
\begin{align} \label{increment}
\PP(|Z_s - Z_t| > \eta) \, \leq \,  A \exp \left( - \frac{\eta^2}{B \delta(s, t)^2} \right).
\end{align}
Then there exists $C = C(A) > 0$ such that for any $\eta > 0$,
\begin{align*}
\PP(\sup_{s, t \in T} |Z_s - Z_{t}| > \eta) \, \leq \, C \exp\left( - \frac{\eta^2}{C B \EE(T, \delta)^2} \right),
\end{align*}
and for any $q \geq 1$, there exists $C_q = C_q(A) > 0$ such that
\begin{align*}
\E \big[ \sup_{s, t \in T} |Z_s - Z_t|^q \big]^{\frac1q} \, \leq \, C_q \sqrt{B} \EE (T, \delta),
\end{align*}
where $\EE (T, \delta)$ is defined by
\beqs
\EE(T, \delta) \, = \, \int_{0}^{\infty} \sqrt{\log N(T, \delta, \epsilon) } \, \d \epsilon.
\eeqs
It is worth noting that if $T$ is not countable, the following condition is required:
\begin{align} \label{uncountable}
\E \big[ \sup_{s, t \in T} |Z_s - Z_t| \big] \, = \, \sup \bigg\{ \E \big[ \sup_{s, t \in F} |Z_s - Z_t| \big] : F \subset T, \,\, \textrm{$F$ is finite}  \bigg\}.
\end{align}
\end{lem}

Let $\|f_{t, r}^{(\alpha)} - 1 \|_\infty = \sup_{y \in M}|f_{t, r}^{(\alpha)}(y) - 1|$. Then the uniform fluctuation bound is achieved.

\begin{propo}[Uniform fluctuation bound] \label{uniform bound'}
There exists $c > 0$ such that for any $t\geq t_M$, $r \in (0, 1]$ and $\eta > 0$,
\beq \label{prob tail}
\sup_{\nu \in \Pr(M)} \PP^\nu \big( \|f_{t, r}^{(\alpha)} - 1 \|_\infty > \eta  \big) \, \leq \, c \exp \left( - \frac{ r^d \eta^2}{c R_\alpha(t) \log(r^{-1} + 1)}   \right),
\eeq
and for any $q \geq 1$, there exists $c_q > 0$ such that
\end{propo}
\beq \label{lq bound}
\sup_{\nu \in \Pr(M)} \E^\nu \big[ \|f_{t, r}^{(\alpha)} - 1 \|_\infty^q \big]^{\frac{1}{q}} \, \leq \, c_q R_\alpha(t)^{\frac12} r^{-\frac{d}{2}} \log(r^{-1} + 1)^{\frac12}.
\eeq

\begin{proof}
For any sample path $(X_u)_{u \in [0, t]}$, $f_{t, r}^{(\alpha)}$ is Lipschitz (therefore continuous) and satisfies $\mu(f_{t, r}^{(\alpha)}) = 1$. Since $M$ is connected, by the mean value theorem we can always find $w \in M$ such that $f_{t, r}^{(\alpha)}(w) = 1$. Consequently,
\beq \label{pointwise}
\|f_{t, r}^{(\alpha)} - 1 \|_\infty \, \le \, \sup_{y, z \in M} |f_{t, r}^{(\alpha)}(y) - f_{t, r}^{(\alpha)}(z)| \, \le \, 2\|f_{t, r}^{(\alpha)} - 1 \|_\infty.
\eeq
Notice that $M$ is separable, the condition \eqref{uncountable} is clearly valid for $(f_{t, r}^{(\alpha)}(y))_{y \in M}$. Concerning \eqref{Lip}, this proposition is an adaptation of Lemma \ref{generic chaining} after we use the fact that for each $u > 0$,
\beq \label{gammau}
\EE(M, \rho_u) \, \le \, c u \sqrt{\log(u^{-1} + 1)}.
\eeq
Indeed, in the setting of the compact Riemannian manifold $(M, \rho)$, by a standard volume argument, it is known that
\beqs
N(M, \rho, \epsilon) \, \le \, \max \{ c \epsilon^{-d}, 1\},
\eeqs
where $d$ is the dimension of $M$. Using the definition of $\rho_u$,
\beqs
N(M, \rho_u, \epsilon ) \, = \,
\begin{cases}
N(M, \rho, \epsilon), & \textrm{if } \epsilon \in (0, u);\\
1, & \textrm{if } \epsilon \in [u, \infty).
\end{cases}
\eeqs
Then for $0 < u < \diam M$,
\beqs
\EE(M, \rho_u) \, = \, \int_{0}^{u} \sqrt{\log N(M, \rho, \epsilon) } \, \d \epsilon \, \le \, c u \sqrt{\log(u^{-1} + 1)},
\eeqs
and for $u \ge \diam M$, $\EE(M, \rho_u)  =  \EE (M, \rho)$, which proves \eqref{gammau}.

\end{proof}
\begin{remark}
An alternative approach to the bound \eqref{lq bound} is based on the hypercontractivity property of $(P_t)_{t > 0}$. Indeed, by \eqref{general sobolev}, for $p \ge 1$, there exists $c > 0$ such that for any $r \in (0, 1]$,
\beqs
\|f_{t, r}^{(\alpha)} - 1 \|_\infty \, = \, \|P_{\frac{r}{2}} (f_{t, \frac{r}{2}}^{(\alpha)} - 1 ) \|_\infty \, \leq \, c  ( r/2)^{-\frac{d}{2p}} \|f_{t, \frac{r}{2}}^{(\alpha)} - 1 \|_p.
\eeqs
Then given $q \ge 1$, for $q \le p < \infty$, by H\"older's inequality,
\beqs
\E^\nu \big[ \|f_{t, r}^{(\alpha)} - 1 \|_\infty^q \big]^{\frac{1}{q}} \, \le \, c  ( r/2)^{-\frac{d}{2p}} \E^\nu \big[ \|f_{t, \frac{r}{2}}^{(\alpha)} - 1 \|_p^p \big]^{\frac{1}{p}}.
\eeqs
Using Fubini's theorem,
\beqs
\E^\nu \big[ \|f_{t, \frac{r}{2}}^{(\alpha)} - 1 \|_p^p \big] \, = \, \int_{M} \E^\nu \big[ |f_{t, \frac{r}{2}}^{(\alpha)}(y) - 1|^p \big] \d \mu(y) \, = \, \int_{M} \E^\nu \big[ |\mu_t^{(\alpha)}(p_{\frac{r}{2}}(\cdot, y) - 1)|^p \big] \d \mu(y).
\eeqs
Combined with the heat kernel bound \eqref{upper pt} and Lemma \ref{fluctuation1}, for any $t\geq t_M$ and $r \in (0, 1]$,
\beqs
\sup_{\nu \in \Pr(M)} \E^\nu \big[ \|f_{t, r}^{(\alpha)} - 1 \|_\infty^q \big]^{\frac{1}{q}} \, \leq \, c R_\alpha(t)^{\frac12} r^{-\frac{d}{2}} \cdot \inf_{p \ge q}\big\{ \sqrt{p} \cdot ( r/2)^{-\frac{d}{2p}}  \big\} .
\eeqs
Optimizing in $p \ge q$ depending on $r \in (0, 1]$ yields that
\beqs
\inf_{p \ge q}\big\{ \sqrt{p} \cdot ( r/2)^{-\frac{d}{2p}}  \big\} \, \le \, c \big(q \vee \log(r^{-1} + 1) \big)^{\frac12},
\eeqs
which establishes \eqref{lq bound}.
\end{remark}

We conclude this section with another application of Lemma \ref{fluctuation1}, which will be used in the next section.
\begin{cor} \label{general p}
For any $q \ge 1$, there exists $c_q > 0$ such that for any $t\geq t_M$ and $r \in (0, 1]$,
\beqs
\sup_{\nu \in \Pr(M)} \E^\nu \big[ \| f_{t, r}^{(\alpha)} - 1 \|_2^{2q} \big]^{\frac 1q}  \, \leq \, c_q   R_\alpha(t) r^{-d}.
\eeqs
\end{cor}

\begin{proof}
Notice that $f_{t, r}^{(\alpha)}(y) - 1 = \mu_t^{(\alpha)}(p_r(\cdot, y) - 1)$. By Minkowski's integral inequality, for any $\nu \in \Pr(M)$,
\begin{align*}
\E^\nu \big[ \| f_{t, r}^{(\alpha)} - 1 \|_2^{2q} \big]^{\frac 1q} & \, = \, \E^\nu \left[ \bigg( \int_M |\mu_t^{(\alpha)}(p_r(\cdot, y) - 1)|^2 \d \mu(y)   \bigg)^q \right]^{\frac 1q}\\
& \, \leq \, \int_M \E^\nu \big[ |\mu_t^{(\alpha)}(p_r(\cdot, y) - 1)|^{2q} \big]^{\frac 1q} \d \mu(y).
\end{align*}
The proof is complete after applying Lemma \ref{fluctuation1} and \eqref{upper pt}.
\end{proof}

\section{Proofs of Theorems \ref{main} and \ref{main1'}} \label{s4}
\setcounter{equation}{0}

\subsection{More about Sobolev norm estimates}

In this sub-section, we shall develop upper bound estimates of $H^{-1, 2p}$-Sobolev norm of $f_{t, r}^{(\alpha)} - 1$, which can be viewed as a generalization of of Corollary \ref{final}. To begin with, we can establish the following lemma, which is an interpolation between Lemma \ref{square bound} and Corollary \ref{general p}.

\begin{lem} \label{J0}
Assume that $M$ is compact with $d \le 3$. Given $\epsilon > 0$. Then for any $1 \le p < \infty$, there exists $c > 0$ such that for any $1 \le p < \infty$, $t\geq t_M$ and $s > 0$,
\beqs
\sup_{\nu \in \Pr(M)}  R_\alpha(t)^{-1} \E^\nu \big[ \| f_{t, s}^{(\alpha)} - 1 \|_2^{2p} \big]^{\frac 1p}  \, \leq \, c \e^{-2\lambda_1 s} \cdot
\begin{cases}
(1 \wedge s)^{-d + \frac{d + 4\alpha}{2p} - \epsilon}, & \textrm{if } \alpha \in (0, \frac{d \wedge 2}{4});\\
(1 \wedge s)^{-d + \frac{d + d \wedge 2}{2p} - \epsilon}, & \textrm{if } \alpha \in [\frac{d \wedge 2}{4}, \infty).
\end{cases}
\eeqs
\end{lem}

\begin{proof}
By Poincar\'e inequality \eqref{Poin}, for $s \ge 1$,
\beqs
\E^\nu \big[ \| f_{t, s}^{(\alpha)} - 1 \|_2^{2p} \big]^{\frac 1p} \, \le \, \e^{-2\lambda_1 (s-1)} \E^\nu \big[ \| f_{t, 1}^{(\alpha)} - 1 \|_2^{2p} \big]^{\frac 1p}.
\eeqs
It thus suffices to prove the case $0 < s < 1$. For simplicity, we set $N_q  =  \E^\nu \big[ \| f_{t, s}^{(\alpha)} - 1 \|_2^{2q} \big]^{\frac 1q}$ where $1 \le q < \infty$. By H\"older's inequality, for any $\theta \in [0, 1)$,
\beqs
N_p \, \le \, N_1^{\frac{\theta}{p}} \cdot N_{\frac{p-\theta}{1 - \theta}}^{1 - \frac{\theta}{p}} \, = \, \big( N_1^{\frac{1}{p}} \cdot N_{\frac{p-\theta}{1 - \theta}}^{1 - \frac{1}{p}} \big) \cdot \big( N_1^{-1} N_{\frac{p-\theta}{1 - \theta}}  \big)^{\frac{1-\theta}{p}}.
\eeqs
Notice that the upper bound on $N_1$ (respectively $N_{\frac{p-\theta}{1 - \theta}}$) has been established in Lemma \ref{square bound} (respectively Corollary \ref{general p}). This lemma follows after choosing a $\theta$ sufficiently close to 1. Note that in some cases we also use the fact that for any $\epsilon > 0$, there exists $c_{\epsilon} > 0$ such that
\beqs
\log (s^{-1} + 1) \, \le \, c_\epsilon s^{-\epsilon}, \quad s \in (0, 1].
\eeqs
\end{proof}

\begin{lem}[$H^{-1, 2p}$-Sobolev norm estimate] \label{J1}
Assume that $M$ is compact with $d \le 3$. If $d \le 2$ and $\alpha \in (0, \infty)$, or $d = 3$ and $\alpha \in (\frac14, \infty)$, there exists a critical exponent $p_0 = p_0(\alpha, d) \in (1, \infty)$ such that for any $1 \leq p < p_0$,
\begin{align} \label{norm}
\sup_{\substack{\nu \in \Pr(M) \\ t \ge t_M, \, r > 0}}  R_\alpha(t)^{-1} \E^\nu\left[ \int_M |\nabla (-\L)^{-1} (f_{t, r}^{(\alpha)} - 1)|^{2p} \d\mu \right]^{\frac1p}  \, < \, \infty.
\end{align}
\end{lem}

\begin{proof}
By the local Poincar\'e inequality, for any $s > 0$ and $f \in C_b(M)$,
\begin{align*}
|\nabla P_s f| &\, \leq \, \sqrt{\frac{K}{\e^{2Ks} - 1}} \, P_s(f^2)^{\frac12} \, \leq \, c  (1 \wedge s)^{-\frac12} P_s(f^2)^{\frac12}.
\end{align*}
Therefore, for $1 \le p < \infty$,
\begin{align} \label{I10}
& \int_M |\nabla (-\L)^{-1} (f_{t, r}^{(\alpha)} - 1)|^{2p} \d\mu  \, \le \,  \int_M \bigg( \int_{0}^{\infty} |\nabla P_s (f_{t, r}^{(\alpha)} - 1)| \d s \bigg)^{2p}   \d \mu  \nonumber\\
& \, \leq \, c  \int_M  \bigg(  \int_{0}^{\infty} (1 \wedge s)^{-\frac12}   P_{\frac{s}{2}} \big( |P_{\frac{s}{2}} (f_{t, r}^{(\alpha)} - 1) |^2 \big)^{\frac 12}  \d s  \bigg)^{2p} \d \mu \nonumber\\
& \, \leq \, c   \bigg( \int_{0}^{\infty} (1 \wedge s)^{-\frac12} \big\| P_{\frac{s}{2}} \big( |P_{\frac{s}{2}} (f_{t, r}^{(\alpha)} - 1) |^2 \big)  \big\|_p^{\frac 12}   \d s \bigg)^{2p},
\end{align}
where in the last line we used Minkowski's integral inequality. Combining \eqref{general sobolev}, \eqref{Poin} and the semigroup property, we obtain
\begin{align} \label{I11}
 \big\| P_{\frac{s}{2}} \big( |P_{\frac{s}{2}} (f_{t, r}^{(\alpha)} - 1) |^2 \big)  \big\|_p^{\frac 12} \, &\le \, \| P_{\frac{s}{2}} \|_{1 \to p}^{\frac 12}  \cdot \big\| P_{\frac{s}{2}} (f_{t, r}^{(\alpha)} - 1) \big\|_2 \, \le \, c  (1 \wedge s)^{-\frac{d}{4p}(p - 1)} \e^{-\frac{\lambda_1 s}{4}} \big\| f_{t, r + \frac{s}{4}}^{(\alpha)} - 1 \big\|_2 \nonumber\\
\, &\le \, c  (1 \wedge s)^{-\frac{d}{4p}(p - 1)} \e^{-\frac{\lambda_1 s}{4}} \cdot \e^{-\lambda_1 r} \big\| f_{t, \frac{s}{4}}^{(\alpha)} - 1 \big\|_2.
\end{align}
Combining \eqref{I10} and \eqref{I11}, by Cauchy-Schwarz inequality,
\begin{align*}
& \int_M |\nabla (-\L)^{-1} (f_{t, r}^{(\alpha)} - 1)|^{2p} \d\mu \\
& \, \le \, c   \bigg( \int_{0}^{\infty} (1 \wedge s)^{-\frac{d}{4p}(p - 1) - \frac12} \e^{-\frac{\lambda_1 s}{4}} \cdot \e^{-\lambda_1 r}  \big\| f_{t, \frac{s}{4}}^{(\alpha)} - 1 \big\|_2   \d s \bigg)^{2p}\\
& \, \le \, c  \bigg(  \int_{0}^{\infty} (1 \wedge s)^{- \kappa} \e^{-\frac{\lambda_1 s}{2}} \d s \bigg)^p \cdot \bigg(  \int_{0}^{\infty} (1 \wedge s)^{-\frac{d}{2p}(p - 1) + \kappa - 1 } \big\| f_{t, \frac{s}{4}}^{(\alpha)} - 1 \big\|_2^2 \, \d s \bigg)^p,
\end{align*}
where $\kappa \in (0, 1)$ will be determined later. Then Minkowski's integral inequality leads to
\beq \label{converge}
\E^\nu\left[ \int_M |\nabla (-\L)^{-1} (f_{t, r}^{(\alpha)} - 1)|^{2p} \d\mu \right]^{\frac1p} \, \leq \, c_\kappa  \int_{0}^{\infty} (1 \wedge s)^{-\frac{d}{2p}(p - 1) + \kappa - 1}  \E^\nu \big[ \big\| f_{t, \frac{s}{4}}^{(\alpha)} - 1 \big\|_2^{2p} \big]^{\frac 1p}  \d s.
\eeq
Using Lemma \ref{J0}, if we set
\begin{align} \label{defofp0}
p_0 \, = \, \begin{cases}
\frac{2d + 4\alpha}{3d - 2}, & \textrm{if }\alpha \in (\frac{(d - 2)^+}{4}, \frac{d \wedge 2}{4});\\
\frac{2d + d \wedge 2}{3d - 2}, & \textrm{if } \alpha \in [\frac{d \wedge 2}{4}, \infty),
\end{cases}
\end{align}
then for $1 \leq p < p_0$ we can always find $\kappa \in (0, 1)$ and sufficiently small $\epsilon > 0$ to guarantee the right-hand side of \eqref{converge} converges, which ends the proof.

\end{proof}

\subsection{Regularization estimates}

In this sub-section, we develop a stronger error bounds on the Wasserstein distance along the heat flow in the context of diffusion processes. The same idea has been used in \cite{AG18, Z21} to deal with the optimal matching problem.

\begin{theo} \label{regularization}
Under the assumptions of Theorem \ref{main}. For $\alpha > 0$, if we choose $r = r(t) = t^{-\beta}$, where $\beta \in (0, \frac{(2\alpha)\wedge 1}{d})$, then for $t \ge t_M$,
\beqs
\sup_{\nu \in \Pr(M)} R_\alpha(t)^{-1} \E^\nu \big[ \W_2^2 \big( \mu_{t, r}^{(\alpha)}, \mu_t^{(\alpha)} \big) \big] \, \le \, 16 R_\alpha(t)^{-1} \int_{0}^{r} \sup_{\nu \in \Pr(M)} \E^\nu \big[ \| f_{t, r}^{(\alpha)} - 1 \|_2^2 \big] \d s \, + \, \Phi(t),
\eeqs
where the function $\Phi : [t_M, \infty) \to [0, \infty)$ decays rapidly to $0$ as $t \to \infty$ in the sense that for any $c > 0$,
$\lim_{t \to \infty} t^c \, \Phi(t) \, = \, 0$.
\end{theo}

\begin{proof}
For $r = r(t) = t^{-\beta}$ with $\beta \in (0, \frac{(2\alpha)\wedge 1}{d})$, we consider the event
\beqs
A_t \, = \, \bigg\{ \| f_{t, r}^{(\alpha)} - 1 \|_\infty \le \frac12   \bigg\}.
\eeqs
On this event, using Lemma \ref{basic wasserstein}, for any $0 < r_* <r$,
\beqs
\W_2^2(\mu_{t, r}^{(\alpha)}, \mu_{t, r_*}^{(\alpha)}) \, \le \, 4 \int_M \frac{|\nabla (-\L)^{-1} (f_{t, r}^{(\alpha)} - f_{t, r_*}^{(\alpha)})|^2}{f_{t, r}^{(\alpha)}} \d \mu \, \le \, 8 \int_M |\nabla (-\L)^{-1} (f_{t, r}^{(\alpha)} - f_{t, r_*}^{(\alpha)})|^2 \d \mu.
\eeqs
Combining with Remark \ref{remark1}, we obtain
\beqs
\W_2^2 \big( \mu_{t, r}^{(\alpha)}, \mu_{t, r_*}^{(\alpha)} \big) \, \le \, 16 \int_{r_*}^{r} \| f_{t, s}^{(\alpha)} - 1 \|^2   \d s.
\eeqs
By Lemma \ref{heat wasserstein1} (i) and the dominated convergence Theorem,
\beqs
\E^\nu \big[ \W_2^2 \big( \mu_{t, r}^{(\alpha)}, \mu_t^{(\alpha)}) \mathbbm{1}_{A_t}  \big] \, = \, \lim_{r_* \downarrow 0} \E^\nu \big[ \W_2^2 \big( \mu_{t, r}^{(\alpha)}, \mu_{t, r_*}^{(\alpha)} ) \mathbbm{1}_{A_t}  \big].
\eeqs
Therefore
\beqs
\sup_{\nu \in \Pr(M)} R_\alpha(t)^{-1} \E^\nu \big[ \W_2^2 \big( \mu_{t, r}^{(\alpha)}, \mu_t^{(\alpha)} \big) \mathbbm{1}_{A_t} \big] \, \le \, 16 R_\alpha(t)^{-1} \int_{0}^{r} \sup_{\nu \in \Pr(M)} \E^\nu \big[ \| f_{t, r}^{(\alpha)} - 1 \|_2^2 \big] \d s.
\eeqs

On the other hand, since $\W_2^2 \big( \mu_{t, r}^{(\alpha)}, \mu_t^{(\alpha)}) \le (\diam M)^2$,
\beqs
\E^\nu \big[\W_2^2 \big( \mu_{t, r}^{(\alpha)}, \mu_t^{(\alpha)} \big) \mathbbm{1}_{A_t^c} \big] \, \le \, c \PP^\nu(A_t^c).
\eeqs
Then we set
$$\Phi(t): = \sup_{\nu \in \Pr(M)} R_\alpha(t)^{-1}\PP^\nu(A_t^c). $$
Let $\upsilon = (2\alpha) \wedge 1 - d \beta > 0$. By Proposition \ref{uniform bound'}, there exists $c > 0$ such that for any $t \ge t_M$,
\beqs
\Phi(t) \, \le \, c R_\alpha(t)^{-1} \exp \left( - \frac{ t^\upsilon}{c \log(t)^2} \right),
\eeqs
which completes the proof.
\end{proof}

\subsection{Upper bound} \label{up}

Throughout this sub-section, for initial distribution $\nu \in \Pr(M)$, we set
\beqs
L_\alpha (\nu) \, = \,
\begin{cases}
\nu ( H^{(\alpha)} ), & \textrm{if } \alpha \in (\frac{(d - 2)^+}{4}, \frac12);\\
\frac12 \sum_{i = 1}^{\infty} \lambda_i^{-2}, & \textrm{if } \alpha = \frac12;\\
\frac{2 \alpha^2}{2\alpha - 1} \sum_{i = 1}^{\infty} \lambda_i^{-2}, & \textrm{if } \alpha \in (\frac12, \infty).
\end{cases}
\eeqs
Then  $\inf_{\nu \in \Pr(M)} L_\alpha (\nu) > 0$. And for $t \ge t_M$ and $r \in (0, 1]$, we abbreviate
\beqs
E_{t, r} \, = \, R_\alpha(t)^{-1} \E^\nu \left[ \int_M |\nabla (-\L)^{-1} (f_{t, r}^{(\alpha)} - 1)|^2 \d\mu \right] - L_\alpha (\nu),
\eeqs
\beqs
F_{t, r} \, = \, \sup_{\nu \in \Pr(M)} \E^\nu \big[ \| f_{t, r}^{(\alpha)} - 1 \|_2^2 \big], \qquad Q_{t, r} \, = \, \int_{0}^{r} F_{t, s} \d s.
\eeqs

For $\varrho \in (0, 1)$, we define $\mu_{t, r, \varrho}^{(\alpha)}$ as follows
\beqs
\mu_{t, r, \varrho}^{(\alpha)} \, := \, (1 - \varrho)\mu_{t, r}^{(\alpha)} + \varrho \mu.
\eeqs
The density function of $\mu_{t, r, \varrho}^{(\alpha)}$ with respect to $\mu$ is $(1 - \varrho) f_{t, r}^{(\alpha)} + \varrho$. By the convexity of $\W_2^2$ (see \cite{V09}) and the fact $\diam M < \infty$, we get
\beq \label{normalize2}
\W_2^2 \big( \mu_{t, r, \varrho}^{(\alpha)}, \mu_{t, r}^{(\alpha)} \big) \, \leq \, \varrho \W_2^2 \big( \mu, \mu_{t, r}^{(\alpha)} \big) \, \leq \, c \varrho.
\eeq

By the triangle inequality of $\W_2$, for $\varepsilon, \varrho \in (0, 1)$ and $r \in (0, 1]$,
\begin{align} \label{triangle}
\W_2^2(\mu_t^{(\alpha)}, \mu)  & \, \leq \, (1 + \varepsilon) \W_2^2(\mu_{t, r, \varrho}^{(\alpha)}, \mu) \, + \, (1 + \varepsilon^{-1}) \W_2^2 \big(\mu_{t, r, \varrho}^{(\alpha)}, \mu_t^{(\alpha)} \big) \nonumber\\
& \, \leq \,  (1 + \varepsilon) \W_2^2(\mu_{t, r, \varrho}^{(\alpha)}, \mu) \, + \, 4 \varepsilon^{-1} \big( \W_2^2 \big(\mu_{t, r, \varrho}^{(\alpha)}, \mu_{t, r}^{(\alpha)} \big) + \W_2^2 \big( \mu_{t, r}^{(\alpha)}, \mu_t^{(\alpha)} \big) \big).
\end{align}

We start with the analysis of $\W_2^2(\mu_{t, r, \varrho}^{(\alpha)}, \mu)$. By Lemma \ref{basic wasserstein}, we yield
\begin{align*}
\W_2^2(\mu_{t, r, \varrho}^{(\alpha)}, \mu)  \, &\leq \, (1 - \varrho)^2 \int_M \frac{|\nabla (-\L)^{-1} (f_{t, r}^{(\alpha)} - 1)|^2}{\M((1 - \varrho)f_{t, r}^{(\alpha)} + \varrho, 1)} \d \mu \, \leq \, \int_M \frac{|\nabla (-\L)^{-1} (f_{t, r}^{(\alpha)} - 1)|^2}{\M((1 - \varrho)f_{t, r}^{(\alpha)} + \varrho, 1)} \d \mu\\
\, &\le \,  \int_M |\nabla (-\L)^{-1} (f_{t, r}^{(\alpha)} - 1)|^2 \cdot \big( 1 + |\M((1 - \varrho)f_{t, r}^{(\alpha)} + \varrho, 1)^{-1} - 1| \big) \d \mu.
\end{align*}
Then using H\"older's inequality and Lemma \ref{J1}, for $\frac{p_0}{p_0 - 1} < q < \infty$ and $t \ge t_M$,
\begin{align*}
R_\alpha(t)^{-1} \E^\nu \big[ \W_2^2(\mu_{t, r, \varrho}^{(\alpha)}, \mu) \big]  &\, \leq \,R_\alpha(t)^{-1} \E^\nu \left[ \int_M |\nabla (-\L)^{-1} (f_{t, r}^{(\alpha)} - 1)|^2 \d \mu \right]\\
& \quad \, + \, c_q \E^\nu \left[ \int_M |\M((1 - \varrho)f_{t, r}^{(\alpha)} + \varrho, 1)^{-1}-1 |^{q} \d \mu \right]^{\frac1q},
\end{align*}
where
\beqs
c_q \, = \, \sup_{\substack{\nu \in \Pr(M) \\ t \ge t_M, \, r > 0}}  R_\alpha(t)^{-1} \E^\nu\left[ \int_M |\nabla (-\L)^{-1} (f_{t, r}^{(\alpha)} - 1)|^{\frac{2q}{q - 1}} \d\mu \right]^{1 - \frac1q}  \, < \, \infty.
\eeqs
If we choose $\varrho = \varrho(t) = t^{-\gamma}$ for some $\gamma > 0$, then using Lemma \ref{fluc of density} for $Y = f_{t, r}^{(\alpha)} \in L^2(\PP^\nu \otimes \mu)$, we obtain
\beq
R_\alpha(t)^{-1} \E^\nu \big[ \W_2^2(\mu_{t, r, \varrho}^{(\alpha)}, \mu) \big] \, \leq \, L_\alpha (\nu) \, + \, E_{t, r} \, + \, c_{q, \gamma} \log(t)  F_{t, r}^{\frac{1}{q\vee2}}.
\eeq

On the other hand, if we choose $r = r(t) = t^{-\beta}$ for some $\beta \in (0, \frac{(2\alpha)\wedge 1}{d})$, then Theorem \ref{regularization} tells us that
\beq \label{normalize3}
R_\alpha(t)^{-1} \E^\nu \big[ \W_2^2 \big( \mu_{t, r}^{(\alpha)}, \mu_t^{(\alpha)} \big) \big] \, \le \, 16 Q_{t, r} \, + \, \Phi(t).
\eeq
Combining \eqref{normalize2}-\eqref{normalize3},
\begin{align*}
R_\alpha(t)^{-1} \E^\nu \big[ \W_2^2(\mu_{t}^{(\alpha)}, \mu) \big] \, &\le \, L_\alpha (\nu) \, + \, c \big( E_{t, r} + c_{q, \gamma} \log(t) F_{t, r}^{\frac{1}{q\vee2}} \big) \, + \,  c (\varepsilon + \varepsilon^{-1} Q_{t, r})\\
  \, &\qquad + \, c \varepsilon^{-1} (R_\alpha(t)^{-1} t^{-\gamma} + \Phi(t)).
\end{align*}
By Corollary \ref{square bound'}, for the preceding choice of $r = r(t)$, one has $\lim_{t \to \infty} Q_{t, r} = 0$. We thus choose $\varepsilon = Q_{t, r}^{1/2}$ and yield
\begin{align*}
R_\alpha(t)^{-1} \E^\nu \big[ \W_2^2(\mu_{t}^{(\alpha)}, \mu) \big] \, - \, L_\alpha (\nu) \, &\le  \, c \big( E_{t, r} \, + \,c_{q, \gamma} \log(t) F_{t, r}^{\frac{1}{q\vee2}} \, + \, Q_{t, r}^{\frac12} \big)\\
\, & \qquad + \, c Q_{t, r}^{-\frac12} (R_\alpha(t)^{-1} t^{-\gamma} \, + \,  \Phi(t)).
\end{align*}
With $\alpha$ and $\beta$ fixed, if $\gamma > 0$ is chosen large enough, the term $Q_{t, r}^{-\frac12} (R_\alpha(t)^{-1} t^{-\gamma} + \Phi(t))$ makes no contribution to the magnitude of the above inequality's right-hand side as $t \to \infty$. It therefore suffices to control
\beqs
K(t) \, := \, \sup_{\nu \in \Pr (M)}| E_{t, r} | \, + \, \log(t) F_{t, r}^{\frac{1}{q\vee2}} \, + \, Q_{t, r}^{\frac12},
\eeqs
where $q > \frac{p_0}{p_0 - 1}$ and $r = r(t) = t^{-\beta}$ for $\beta \in (0, \frac{(2\alpha)\wedge 1}{d})$.

Next we take the case $d = 1$, $\alpha > \frac12$ as an example. The upper bounds of $\sup_{\nu \in \Pr (M)}| E_{t, r} |$ (respectively, $F_{t, r}$ and $Q_{t, r}$) have been obtained in Proposition \ref{energy} and \ref{energy'} (respectively, Corollary \ref{square bound'}). And by \eqref{defofp0},  in this case, the critical exponent $p_0 = 3 > 2$. Then it is immediate that for $t \ge t_M$ and $\frac32 < q \le 2$,
\beqs
\sup_{\nu \in \Pr (M)}| E_{t, r} | \, \le \,c (t^{-\beta} + t^{-2(\alpha \wedge 1)+1}), \quad \log(t) F_{t, r}^{\frac{1}{q\vee2}} \, \le \, c \log(t) t^{-\frac12}, \quad Q_{t, r}^{\frac12} \, \le \, c t^{-\frac{\beta}{2}}, \quad \beta \in (0, 1).
\eeqs
Optimizing in $\beta$ implies that
\beqs
K(t) \, = \,
\begin{cases}
O(t^{-2\alpha + 1}), & \textrm{if } \alpha \in (\frac12, \frac34);\\
O_\epsilon(t^{-\frac12 + \epsilon}), & \textrm{if } \alpha \in [\frac34, \infty).
\end{cases}
\eeqs
Dealing the other cases in a similar way, it finally follows that
\begin{enumerate}[(i)]
\item when $d = 1$,
\beqs
K(t) \, = \,
\begin{cases}
O_\epsilon(t^{-2\alpha^2 - \frac{\alpha}{2}+\epsilon}), & \textrm{if } \alpha \in (0, \frac14];\\
O_\epsilon(t^{-\alpha + \epsilon}), & \textrm{if } \alpha  \in (\frac14, \frac13];\\
O(t^{-|2\alpha - 1|}), & \textrm{if } \alpha \in (\frac13, \frac34) \backslash \{ \frac12 \};\\
O(\log(t)^{-1}), & \textrm{if } \alpha  = \frac12;\\
O_\epsilon(t^{-\frac12 + \epsilon}), & \textrm{if } \alpha \in [\frac34, \infty);
\end{cases}
\eeqs

\item when $d = 2$,
\beqs
K(t) \, = \,
\begin{cases}
O_\epsilon(t^{- \alpha^2 + \epsilon}), & \textrm{if } \alpha \in (0, \sqrt{2}-1];\\
O(t^{-|2\alpha - 1|}), & \textrm{if } \alpha \in (\sqrt{2}-1, \frac58) \backslash \{ \frac12 \};\\
O(\log(t)^{-1}), & \textrm{if } \alpha  = \frac12;\\
O_\epsilon(t^{-\frac14 + \epsilon}), & \textrm{if } \alpha \in [\frac58, \infty);
\end{cases}
\eeqs

\item when $d = 3$,
\beqs
K(t) \, = \,
\begin{cases}
O_\epsilon(t^{-\frac{2\alpha^2}{3} + \frac{\alpha}{6} + \epsilon}), & \textrm{if } \alpha \in (\frac14, \frac{\sqrt{217}-11}{8}];\\
O(t^{-|2\alpha - 1|}), & \textrm{if } \alpha \in (\frac{\sqrt{217}-11}{8}, \frac{13}{24}) \backslash \{ \frac12 \};\\
O(\log(t)^{-1}), & \textrm{if } \alpha  = \frac12;\\
O_\epsilon(t^{-\frac1{12} + \epsilon}), & \textrm{if } \alpha \in [\frac{13}{24}, \infty),
\end{cases}
\eeqs
\end{enumerate}
which completes the proof for the upper bound part.

\subsection{Lower bound} \label{low}
By Lemma \ref{heat wasserstein1} (ii), we have
\beq \label{heat contract}
\W_2^2(\mu_t^{(\alpha)}, \mu) \ge \e^{2Kr} \W_2^2(\mu_{t, r}^{(\alpha)}, \mu).
\eeq
It is therefore sufficient to evaluate $\W_2^2(\mu_{t, r}^{(\alpha)}, \mu)$.

By the Kantorovich duality (cf. \cite{V09}),
\beqs
\frac 12 \W_2^2(\mu_{t, r}^{(\alpha)}, \mu) \, \ge \, \sup_{(\varphi_0, \varphi_1) \in \mathscr{C}} \big\{ \mu_{t, r}^{(\alpha)}(\varphi_0) \, + \, \mu(\varphi_1) \big\},
\eeqs
where
\beqs
\mathscr{C} \, = \, \bigg\{ (\varphi_0, \varphi_1)  :  \varphi_0, \varphi_1 \in C_b(M), \, \varphi_0(x) + \varphi_1(y) \le \frac{\rho(x, y)^2}{2} \textrm{ for any } x, y \in M   \bigg\}.
\eeqs
Since $\d \mu_{t, r}^{(\alpha)} = f_{t, r}^{(\alpha)} \d \mu$, we obtain
\beqs
\frac 12 \W_2^2(\mu_{t, r}^{(\alpha)}, \mu) \, \ge \, \sup_{(\varphi_0, \varphi_1) \in \mathscr{C}} \bigg\{ \int_M \varphi_0 (f_{t, r}^{(\alpha)} - 1) \d \mu  + \int_M (\varphi_0 + \varphi_1) \d \mu  \bigg\}.
\eeqs
For $\varphi_0 = (-\L)^{-1} (f_{t, r}^{(\alpha)} - 1)$, we take
\beqs
\varphi_1 (x) \, : = \, \inf_{y \in M}\bigg\{-\varphi_0(y) + \frac{\rho(x, y)^2}{2}  \bigg\}, \quad x \in M.
\eeqs
Then applying Corollary 3.3 together with (3.3) of \cite{AMB}, there exists $c_* > 0$ such that
\beq \label{Dual}
\frac 12 \W_2^2(\mu_{t, r}^{(\alpha)}, \mu) \, \geq \, \bigg( 1 - \frac{\exp \big(c_* \big( \| f_{t, r}^{(\alpha)} - 1 \|_\infty + \| f_{t, r}^{(\alpha)} - 1 \|_\infty^2 \big) \big)}{2}  \bigg) \int_M |\nabla (-\L)^{-1} (f_{t, r}^{(\alpha)} - 1)|^2 \d \mu.
\eeq
Without loss of generality, we may assume $c_* \ge 1$. Now we reformulate \eqref{Dual} in the form of the expectation. Consider the following event
\beqs
F \, = \, \big\{ \| f_{t, r}^{(\alpha)} - 1 \|_\infty \, \leq \, c_*^{-1} 2^{-1} \big\}.
\eeqs
On $F$, by \eqref{Dual} and a Taylor expansion, there exists $c > 0$ such that
\begin{align} \label{Part1}
\W_2^2(\mu_{t, r}^{(\alpha)}, \mu)
\,  \geq \, \int_M |\nabla (-\L)^{-1} (f_{t, r}^{(\alpha)} - 1)|^2 \d \mu \, - \, c \| f_{t, r}^{(\alpha)} - 1 \|_\infty \int_M |\nabla (-\L)^{-1} (f_{t, r}^{(\alpha)} - 1)|^2 \d \mu.
\end{align}
We abbreviate
\beqs
U \, = \, \| f_{t, r}^{(\alpha)} - 1 \|_\infty, \qquad V \, = \, \int_M |\nabla (-\L)^{-1} (f_{t, r}^{(\alpha)} - 1)|^2 \d \mu.
\eeqs
Then
\begin{align} \label{lower bound}
& \E^\nu \big[ \W_2^2(\mu_{t, r}^{(\alpha)}, \mu) \big] \, \ge \, \E^\nu \big[ \W_2^2(\mu_{t, r}^{(\alpha)}, \mu) \mathbbm{1}_F \big] \, \ge \, \E^\nu[V  \mathbbm{1}_F] - c \E^\nu[ U V \mathbbm{1}_{F}] \nonumber\\
\, &= \, \E^\nu[V] -  \E^\nu[V \mathbbm{1}_{F^c}] - c \E^\nu \big[ U V  \mathbbm{1}_{F}\big] \, \ge \, \E^\nu[V] - c \E^\nu \big[ ( U \wedge 1 ) V  \big]\nonumber\\
\, & \ge \, \E^\nu[V]  - c \E^\nu \big[ U V \big].
\end{align}
Combined with H\"older's inequality and Lemma \ref{J1}, it follows that for $1 < p < p_0$,
\begin{align*}
R_\alpha(t)^{-1} \E^\nu \big[ \W_2^2(\mu_{t, r}^{(\alpha)}, \mu) \big] \, &\ge \, R_\alpha(t)^{-1} \E^\nu[V] \, - \, c_p \E^\nu  \big[ U^{\frac{p}{p-1}} \big]^{1 - \frac1p}\\
\, & \ge \, L_\alpha(\nu) \, + \, E_{t, r} \, - \, c_p \E^\nu  \big[ U^{\frac{p}{p-1}} \big]^{1 - \frac1p},
\end{align*}
where
\beqs
c_p \, = \, \sup_{\substack{\nu \in \Pr(M) \\ t \ge t_M, \, r > 0}}  R_\alpha(t)^{-1} \E^\nu\left[ \int_M |\nabla (-\L)^{-1} (f_{t, r}^{(\alpha)} - 1)|^{2p} \d\mu \right]^{\frac1p}  \, < \, \infty.
\eeqs
Combining this estimate with \eqref{heat contract}, since $\e^{Kr} = 1 + Kr + o(r)$ as $r \to 0$, we obtain
\beq
R_\alpha(t)^{-1} \E^\nu \big[ \W_2^2(\mu_t^{(\alpha)}, \mu) \big] \, \ge \, L_\alpha(\nu) \, - \, c \big( r + \sup_{\nu \in \Pr (M)} | E_{t, r} | + \sup_{\nu \in \Pr (M)} \E^\nu  \big[ \| f_{t, r}^{(\alpha)} - 1 \|_\infty^q \big]^{\frac1q} \big),
\eeq
where $\frac{p_0}{p_0 - 1} < q < \infty$ with $p_0$ given by \eqref{defofp0}. If we choose $r = r(t) = t^{-\zeta}$ for some $\zeta > 0$, it is therefore sufficient to control
\beq
\widetilde{K}(t) \, := \, r + \sup_{\nu \in \Pr (M)} | E_{t, r} | + \sup_{\nu \in \Pr (M)} \E^\nu  \big[ \| f_{t, r}^{(\alpha)} - 1 \|_\infty^q \big]^{\frac1q},
\eeq
where $q > \frac{p_0}{p_0 - 1}$. Combining with Proposition \ref{energy}, \ref{energy'}, \ref{uniform bound'} and then optimizing in $\zeta$, we conclude that
\begin{enumerate}[(i)]
\item when $d = 1$,
\beqs
\widetilde{K}(t) \, = \,
\begin{cases}
O(t^{-\frac{4\alpha^2 + \alpha}{4\alpha + 2}}\log(t)^{\frac12}), & \textrm{if } \alpha \in (0, \frac14);\\
O(t^{-\frac16}\log(t)), & \textrm{if } \alpha = \frac14;\\
O(t^{-\frac{2\alpha}{3}}\log(t)^{\frac12}), & \textrm{if } \alpha \in (\frac14, \frac38];\\
O(t^{-|2\alpha - 1|}), & \textrm{if } \alpha \in (\frac38, \frac23)\backslash \{ \frac12 \};\\
O(\log(t)^{-1}), & \textrm{if } \alpha  = \frac12;\\
O(t^{-\frac13} \log(t)^{\frac12}), & \textrm{if } \alpha \in [\frac23, \infty);
\end{cases}
\eeqs

\item when $d = 2$,
\beqs
\widetilde{K}(t) \, = \,
\begin{cases}
O(t^{-\frac{2\alpha^2}{2\alpha + 1}} \log(t)), & \textrm{if } \alpha \in (0, \frac{\sqrt{6}}{6}];\\
O(t^{-|2\alpha - 1|}), & \textrm{if } \alpha \in (\frac{\sqrt{6}}{6}, \frac58)\backslash \{ \frac12 \};\\
O(\log(t)^{-1}), & \textrm{if } \alpha  = \frac12;\\
O(t^{-\frac14} \log(t)), & \textrm{if } \alpha \in [\frac58, \infty);
\end{cases}
\eeqs

\item when $d = 3$,
\beqs
\widetilde{K}(t) \, = \,
\begin{cases}
O(t^{-\frac{4\alpha^2 - \alpha}{4\alpha + 2}}\log(t)^{\frac12}), & \textrm{if } \alpha \in (\frac14, \frac{1+\sqrt{97}}{24}];\\
O(t^{-|2\alpha - 1|}), \quad & \textrm{if } \alpha \in (\frac{1+\sqrt{97}}{24}, \frac{9}{16})\backslash \{ \frac12 \};\\
O(\log(t)^{-1}), & \textrm{if } \alpha  = \frac12;\\
O(t^{-\frac18}\log(t)^{\frac12}), & \textrm{if } \alpha \in [\frac{9}{16}, \infty).
\end{cases}
\eeqs

\end{enumerate}

\section{Proofs of Theorems \ref{main2'} and \ref{main2}} \label{s5}
\setcounter{equation}{0}

\subsection{Proof of Theorem \ref{main2'}}

By the triangle inequality, Lemma \ref{heat wasserstein1} (i) and Lemma \ref{basic wasserstein}, for any $r \in (0, 1]$,
\begin{align*}
\E^\nu \big[ \W_2^2 ( \mu_t^{(\alpha)}, \mu )  \big] \, &\leq \, 2 \E^\nu \big[ \W_2^2 ( \mu_t^{(\alpha)}, \mu_{t, r}^{(\alpha)} )  \big] \, + \, 2 \E^\nu \big[ \W_2^2 ( \mu_{t, r}^{(\alpha)}, \mu  )  \big]\\
\, & \leq \, c r \, + \, 4 \E^\nu \left[ \int_M |\nabla (-\L)^{-1} (f_{t, r}^{(\alpha)} - 1)|^2 \d \mu  \right].
\end{align*}
Combined this inequality with Remark \ref{remark0}, it is known that
\begin{align*}
\E^\nu \big[ \W_2^2 ( \mu_t^{(\alpha)}, \mu )  \big] \, & \leq \, c r \,  + \, c t^{-2\alpha}  \iint_{\{0 \, \leq \, u \, \leq \, v \, \leq \, t  \}}  \e^{-\lambda_1 u} \| h_{v - u + 2r}^{(1)} \|_\infty \, u^{\alpha - 1} v^{\alpha - 1} \d u   \d v\\
\, & \qquad + \, c
\begin{cases}
t^{-2\alpha} \overline{h}^{(2\alpha + 1)}_{2r}, & \textrm{if } \alpha \in (0, \frac12);\\
t^{-1} \big( \ln t \cdot \overline{h}^{(2)}_{2r}  \, + \, |g^{(2)}_{2r}| \big), & \textrm{if } \alpha = \frac12;\\
t^{-1}\overline{h}^{(2)}_{2r}, & \textrm{if } \alpha \in (\frac12, \infty).
\end{cases}
\end{align*}
Then applying Lemma \ref{trace} and optimizing in $r \in (0, 1]$, we conclude the proof.

\subsection{Proof of Theorem \ref{main2}}

Before we give the proof of Theorem \ref{main2}, we state some useful lemmas, whose proofs are almost the same as those presented in the proceeding sections. To avoid repetition, we only sketch them. In the following up, we always assume $\alpha \in (\frac12, \infty)$.

By the method in Lemma \ref{square bound} and Proposition \ref{energy}, the estimates of $L^2$-norm and $H^{-1, 2}$-Sobolev norm of $f_{t, r}^{(\alpha)} - 1$ are achieved as follows.

\begin{lem} \label{l2}
Assume that $M$ is compact with dimension $d  = 4$. Then there exists $c > 0$ such that for any $t \ge t_M$ and $r \in (0, 1]$,
\begin{enumerate}[(i)]
\item
\beqs
\sup_{\nu \in \Pr(M)} \E^\nu \big[ \| f_{t, r}^{(\alpha)} - 1 \|_2^2 \big]  \, \leq \, c t^{-1} r^{-1}.
\eeqs

\item
\begin{align*}
 \sup_{\nu \in \Pr(M)} & \bigg|  \E^\nu \left[ \int_M |\nabla (-\L)^{-1} (f_{t, r}^{(\alpha)} -  1)|^2 \d\mu \right] \, - \,  \frac{2\alpha^2}{2\alpha - 1} \cdot t^{-1} \sum_{i = 1}^{\infty} \lambda_i^{-2} \e^{-2r\lambda_i} \bigg|\\
\quad \,& \leq \, c \big( t^{-2\alpha} \log(r^{-1} + 1) \, + \,  t^{-2(\alpha \wedge 1)} \big).
\end{align*}
\end{enumerate}
\end{lem}

\begin{proof}
Using the notations of Lemma \ref{square bound} and Proposition \ref{energy}, we have
\beqs
\sup_{\nu \in \Pr(M)} \E^\nu \big[ \| f_{t, r}^{(\alpha)} - 1 \|_2^2 \big]  \, \leq \, c \left( t^{-1}  \overline{h}^{(1)}_{2r} \, + \, t^{-2 \alpha}  \iint_{\{0 \, \leq \, u \, \leq \, v \, \leq \, t  \}}  \e^{- \lambda_1 u} \|  h_{v - u + 2r}^{(0)} \|_\infty \,  \d u   \d v \right),
\eeqs
and
\begin{align*}
& \sup_{\nu \in \Pr(M)} \bigg| \E^\nu \left[ \int_M |\nabla (-\L)^{-1} (f_{t, r}^{(\alpha)} - 1)|^2 \d\mu \right] \, - \, \frac{2\alpha^2}{2\alpha - 1} \cdot t^{-1} \sum_{i = 1}^{\infty} \lambda_i^{-2} \e^{-2r\lambda_i} \bigg|\\
& \, \leq \, c t^{-2\alpha} \iint_{\{0 \, \leq \, u \, \leq \, v \, \leq \, t  \}}  \e^{-\lambda_1 u} \| h_{v - u + 2r}^{(1)} \|_\infty  \d u   \d v \, + \, c
\begin{cases}
t^{-2\alpha} \overline{h}_{2r}^{(2\alpha + 1)} , & \textrm{if }\alpha \in (\frac12, 1);\\
t^{-2} \overline{h}_{2r}^{(3)}, & \textrm{if }\alpha \in [1, \infty).
\end{cases}.
\end{align*}
Combining this with Lemma \ref{trace}, we complete the proof.
\end{proof}

Another ingredient we shall use is the fluctuation bounds proved in Section \ref{s3}. Note that Proposition \ref{uniform bound'} and Corollary \ref{general p} work for any dimension $d$. As a result we have the following lemma.
\begin{lem} \label{con}
Assume that $M$ is compact with dimension $d  = 4$. Then for any $q \ge 1$, there exists $c_q > 0$ such that for any $t \ge t_M$ and $r \in (0, 1]$,
\begin{enumerate}[(i)]
\item
\beqs
\sup_{\nu \in \Pr(M)} \E^\nu \big[ \| f_{t, r}^{(\alpha)} - 1 \|_\infty^q \big]^{\frac1q}  \, \leq \, c_q t^{-\frac12} r^{-2} \log(r^{-1} + 1)^{\frac12}.
\eeqs

\item
\beqs
\sup_{\nu \in \Pr(M)} \E^\nu \big[ \| f_{t, r}^{(\alpha)} - 1 \|_2^{2q} \big]^{\frac1q}  \, \leq \, c_q t^{-1} r^{-4}.
\eeqs

\end{enumerate}

\end{lem}

Interpolating between Lemma \ref{l2} (ii) and Lemma \ref{con} (ii), by the argument in Lemma \ref{J0}, the following lemma is established.
\begin{lem} \label{inter}
Assume that $M$ is compact with $d  = 4$. Given $\epsilon > 0$. Then for any $1 \le p < \infty$, there exists $c > 0$ such that for any $1 \le p < \infty$, $t\geq t_M$ and $s > 0$,
\beqs
\sup_{\nu \in \Pr(M)}  \E^\nu \big[ \| f_{t, s}^{(\alpha)} - 1 \|_2^{2p} \big]^{\frac1p}  \, \leq \, c \e^{-2\lambda_1 s} t^{-1} (1 \wedge s)^{-4 + \frac3p - \epsilon}.
\eeqs
\end{lem}

The following proposition is obtained by a simple modification of the proof of Lemma \ref{J1}.
\begin{propo} \label{H1p bound}
Assume that $M$ is compact with $d  = 4$. Given $\epsilon > 0$. Then for any $1 \le p < \infty$, there exists $c > 0$ such that for any $1 \le p < \infty$, $t\geq t_M$ and $r \in (0, 1]$,
\beqs
\sup_{\nu \in \Pr(M)} \E^\nu\left[ \int_M |\nabla (-\L)^{-1} (f_{t, r}^{(\alpha)} - 1)|^{2p} \d\mu \right]^{\frac1p}  \, \leq \, c t^{-1} r^{-5 + \frac5p - \epsilon}.
\eeqs
\end{propo}

\begin{proof}
Notice that for $1 \leq p < \infty$, combining \eqref{I10} and the semigroup property,
\begin{align*}
& \int_M |\nabla (-\L)^{-1} (f_{t, r}^{(\alpha)} - 1)|^{2p} \d\mu  \, \le \,  \int_M \bigg( \int_{r/2}^{\infty} |\nabla P_s (f_{t, \frac{r}{2}}^{(\alpha)} - 1)| \d s \bigg)^{2p}   \d \mu  \\
& \, \leq \, c   \bigg( \int_{r/2}^{\infty} (1 \wedge s)^{-\frac12} \big\| P_{\frac{s}{2}} \big( |P_{\frac{s}{2}} (f_{t, \frac{r}{2}}^{(\alpha)} - 1) |^2 \big)  \big\|_p^{\frac 12}   \d s \bigg)^{2p}.
\end{align*}
Along the same lines as the argument in Lemma \ref{J1}, we obtain
\beqs
\E^\nu\left[ \int_M |\nabla (-\L)^{-1} (f_{t, r}^{(\alpha)} - 1)|^{2p} \d\mu \right]^{\frac1p} \, \leq \, c_\kappa  \int_{r/2}^{\infty} (1 \wedge s)^{-3 + \frac{2}{p} + \kappa }  \E^\nu \big[ \big\| f_{t, \frac{s}{4}}^{(\alpha)} - 1 \big\|_2^{2p} \big]^{\frac 1p}  \d s,
\eeqs
where $\kappa < 1$ to be specified later. Then this proposition follows by applying Lemma \ref{inter} for $\frac{\epsilon}{2}$ and choosing $\kappa = 1 - \frac{\epsilon}{2}$.

\end{proof}

Now we turn to the proof of Theorem \ref{main2} by using the argument given in Section \ref{low}. We start with \eqref{lower bound}.
\beqs
\E^\nu \big[ \W_2^2(\mu_{t, r}^{(\alpha)}, \mu^{(\alpha)}) \big] \, \ge \,  \E^\nu \left[ \int_M |\nabla (-\L)^{-1} (f_{t, r}^{(\alpha)} - 1)|^2 \d\mu \right] \, - \, c \, \Theta^\nu(t, r),
\eeqs
where
\beqs
\Theta^\nu(t, r) \, := \, \E^\nu \left[ \| f_{t, r}^{(\alpha)} - 1 \|_\infty \cdot \int_M |\nabla (-\L)^{-1} (f_{t, r}^{(\alpha)} - 1)|^2 \d\mu  \right].
\eeqs
For the first term on the right-hand side of this inequality, by Lemma \ref{l2}, there exists $c > 0$ such that for $t \ge t_M$ and $r \in (0, 1]$,
\begin{align} \label{t1}
& \inf_{\nu \in \Pr(M)} \E^\nu \left[ \int_M |\nabla (-\L)^{-1} (f_{t, r}^{(\alpha)} - 1)|^2 \d\mu \right]\nonumber\\
\, &\ge \, \frac{2\alpha^2}{2\alpha - 1} \cdot t^{-1} \sum_{i = 1}^{\infty} \lambda_i^{-2} \e^{-2r\lambda_i} \, - \, c \big( t^{-2\alpha} \log(r^{-1} + 1) \, + \,  t^{-2(\alpha \wedge 1)} \big).
\end{align}
By the small time asymptotics of the trace (see \cite[Theorem 1.5]{CR19}),
\beqs
\lim_{s \to 0} (4 \pi s)^{d/2} \sum_{i = 1}^{\infty} \e^{-s \lambda_i} \, = \, 1,
\eeqs
and setting $d =4$, we have
\beq \label{t2}
\sum_{i = 1}^{\infty} \lambda_i^{-2} \e^{-2r\lambda_i} \, = \, \int_{2r}^{\infty} \bigg(\sum_{i = 1}^{\infty} \e^{-s \lambda_i} \bigg) s \d s \, = \, \frac{\ln (r^{-1})}{16 \pi^2} (1 \, + \, o(1)), \quad \text{as } r \to 0.
\eeq
Then it remains to bound $\Theta^\nu(t, r)$ from above. To reach this goal, by H\"older's inequality and Proposition \ref{H1p bound}, for any $1 \le p < \infty$,
\beqs
\Theta^\nu(t, r) \, = \, \E^\nu \left[ \| f_{t, r}^{(\alpha)} - 1 \|_\infty \cdot \int_M |\nabla (-\L)^{-1} (f_{t, r}^{(\alpha)} - 1)|^2 \d\mu  \right] \, \le \, \Theta_p^\nu(t, r),
\eeqs
where
\beqs
\Theta_p^\nu(t, r) \, := \, \E^\nu \big[ \| f_{t, r}^{(\alpha)} - 1 \|_\infty^\frac{p}{p-1} \big]^{1 - \frac1p} \cdot \E^\nu\left[ \int_M |\nabla (-\L)^{-1} (f_{t, r}^{(\alpha)} - 1)|^{2p} \d\mu \right]^{\frac1p}.
\eeqs
By Lemma \ref{con} (i) and Proposition \ref{H1p bound}, it follows that for any $\epsilon > 0$, there exists $c_{p, \epsilon} > 0$ such that for $t \ge t_M$ and $r \in (0, 1]$,
\beqs
\sup_{\nu \in \Pr(M)} \Theta_p^\nu(t, r) \, \leq \, c_{p, \epsilon} t^{-\frac32} r^{-7 + \frac{5}{p} - \epsilon}.
\eeqs
Combined with \eqref{t1}, it follows that
\begin{align*}
& \inf_{\nu \in \Pr(M)} \E^\nu \big[ \W_2^2(\mu_{t, r}, \mu) \big]\\
\, &\ge \, \frac{2\alpha^2}{2\alpha - 1} \cdot t^{-1} \sum_{i = 1}^{\infty} \lambda_i^{-2} \e^{-2r\lambda_i} \, - \, c \big( t^{-2\alpha} \log(r^{-1} + 1) \, + \,  t^{-2(\alpha \wedge 1)} \big) \, - \,  c_{p, \epsilon} t^{-\frac32} r^{-7 + \frac{5}{p} - \epsilon}.
\end{align*}
For any $\varsigma \in (0, \frac14)$, we can always find $1 < p' < \infty$ and $\epsilon' > 0$ such that
\begin{align*}
\vartheta \, : =  \, \frac12 \, - \, \varsigma \left( 7 - \frac{5}{p'} + \epsilon' \right) \, > \, 0.
\end{align*}
Now we choose $r = r(t) = t^{-\varsigma}$. Then
\begin{align} \label{t3}
& \inf_{\nu \in \Pr(M)} \E^\nu \big[ \W_2^2(\mu_{t, r}, \mu) \big]\nonumber\\
\, &\ge \, \frac{2\alpha^2}{2\alpha - 1} \cdot t^{-1} \sum_{i = 1}^{\infty} \lambda_i^{-2} \e^{-2r\lambda_i} \, - \, c_\varsigma \big( t^{-2\alpha} \log(r^{-1} + 1) \, + \,  t^{-2(\alpha \wedge 1)} \, + \,  t ^{-1 - \vartheta} \big).
\end{align}

Finally, combining \eqref{t2}, \eqref{t3} and the fact that $\W_2^2(\mu_t, \mu) \ge \e^{2Kr} \W_2^2(\mu_{t, r}, \mu)$, it follows that
\beqs
\liminf_{t \to \infty} \frac{t}{\ln t}  \inf_{\nu \in \Pr(M)} \E^\nu \big[ \W_2^2(\mu_t, \mu) \big]  \, \ge \, \frac{\alpha^2}{2\alpha - 1}\frac{\varsigma}{8 \pi^2}.
\eeqs
Since $\varsigma \in (0, \frac14)$ is arbitrary, Theorem \ref{main2} is established.

\paragraph{Acknowledgement.} The author would like to thank Prof. Feng-Yu Wang for helpful discussions and Prof. Michel Ledoux for instructive comments during the preparation of this work.

\end{document}